\documentclass[11pt]{amsart} 
\usepackage[utf8x]{inputenc}

\usepackage{pgf} 

\usepackage{graphicx}
\usepackage{adjustbox}

 \newcommand{\commentgray}[1]{}  

\usepackage{enumerate} 
\usepackage[pdfencoding=auto,hidelinks]{hyperref}
\usepackage{amsfonts,amsmath,amsthm} 
\usepackage{txfonts}
\usepackage{fullpage}

\usepackage{multirow}
\usepackage{float}

\DeclareMathOperator{\wilfoper}{W} 
\DeclareMathOperator{\eliahouoper}{E} 
\DeclareMathOperator{\multiplicityoper}{m} 
\DeclareMathOperator{\conductoroper}{c} 
\DeclareMathOperator{\genusoper}{genus} 
\DeclareMathOperator{\primitivesoper}{P} 
\DeclareMathOperator{\leftsoper}{L} 
\DeclareMathOperator{\qoper}{q} 
\DeclareMathOperator{\rhooper}{\rho} 
\DeclareMathOperator{\Dqoper}{D_{\qoper}} 
\newtheorem{theorem}{Theorem} 
\newtheorem{lemma}[theorem]{Lemma}
\newtheorem{corollary}[theorem]{Corollary}
\newtheorem{proposition}[theorem]{Proposition}
\newtheorem{conjecture}[theorem]{Conjecture}

\theoremstyle{remark} 
\newtheorem{example}[theorem]{Example}
\newtheorem{remark}[theorem]{Remark}
\newtheorem{question}[theorem]{Question}

\title{On a question of Eliahou and a conjecture of Wilf}
\author{Manuel Delgado}
\date{\today}
\address{CMUP, Departamento de Matem\'atica, Faculdade de
  Ci\^encias, Universidade do Porto, Rua do Campo Alegre 687, 4169-007 Porto,
  Portugal} 
\email{mdelgado@fc.up.pt} 
\thanks{The author was partially supported by CMUP (UID/MAT/00144/2013), which is funded by FCT (Portugal) with national (MCTES) and European structural funds (FEDER), under the partnership agreement PT2020, and also by the project MTM2014-55367-P.}

\begin{document}

\begin{abstract}
To a numerical semigroup~$S$, Eliahou associated a number $\eliahouoper(S)$ and proved that numerical semigroups for which the associated number is non negative satisfy Wilf's conjecture. The search for counterexamples for the conjecture of Wilf is therefore reduced to semigroups which have an associated negative Eliahou number. Eliahou mentioned $5$ numerical semigroups whose Eliahou number is $-1$. The examples were discovered by Fromentin who observed that these are the only ones with negative Eliahou number among the over $10^{13}$ numerical semigroups of genus up to $60$. We prove here that for any integer $n$ there are infinitely many numerical semigroups~$S$ such that $\eliahouoper(S)=n$, by explicitly giving families of such semigroups. We prove that all the semigroups in these families satisfy Wilf's conjecture, thus providing not previously known examples of semigroups for which the conjecture holds.
\end{abstract}

\maketitle

\section{Motivation}
This work was motivated by a recent paper by Eliahou~\cite{eliahou}, which constitutes a remarkable contribution towards a better understanding of Wilf's conjecture. Much about the conjecture of Wilf, so as references to the many papers devoted to the subject, can be found in Eliahou's paper.

To each numerical semigroup~$S$, Eliahou associated a number, which we denote by $\eliahouoper(S)$ (Eliahou's notation was $\wilfoper_0(S)$), and proposed as an interesting problem the characterization of the class ${\mathcal E}=\{S\mid \eliahouoper(S)< 0\}$ of numerical semigroups. Suggesting such a problem is not surprising, since he proved that the set $\{S\mid \eliahouoper(S)\ge 0\}$ consists of semigroups that satisfy Wilf's conjecture. Furthermore, he made use of that number to prove that all the numerical semigroups whose conductor is not bigger than the triple of the multiplicity satisfy Wilf's conjecture, which is a great result.

Eliahou observed that it seems to be very rare that a numerical semigroup belongs to ${\mathcal E}$, where, as observed, any possible counterexample to Wilf's conjecture must belong. Despite giving infinite families of semigroups in ${\mathcal E}$, we tend to agree with him. A (pseudo-) random search through all the numerical semigroups would hardly give an example. According to Eliahou, Fromentin discovered the first $5$ examples through exhaustive search among the numerical semigroups of genus up to $60$, which are more than $10^{13}$. This shows that the probability of a numerical semigroup taken at random from the set of numerical semigroups of genus up to $60$ to have negative Eliahou number is approximately $5\cdot 10^{-13}$.

Our strategy has been to do a pseudo-random search, but making some naive guesses that allowed us to highly reduce the search space or even to discard huge amounts of candidates without even looking at them. 
For some of the guesses we were able to sketch simple proofs that no candidate was wrongly rejected, for others not so simple and for many others we have not even convinced ourselves that a proof can be done. Notice that in our strategy this is not an important point (we are not even concerned with some kind of uniform distribution): once we find an example we can easily verify that it is in fact an example.

We observe that the $5$ available examples have been crucial for the initial guesses we made. Nevertheless, finding the first not known examples has not been an easy task, and crucial role was played by computational tools, either in terms of hardware (the search for examples used a large number of hours in several computers) or in terms of software (the \textsf{GAP}~\cite{GAP4} package \textsf{numericalsgps}~\cite{numericalsgps} has been used). 
After, having a large amount of examples at our disposal we looked for patterns. The \textsf{GAP}~\cite{GAP4} package \textsf{intpic}~\cite{intpic} played an important role in this part, by allowing us to have an automatic pictorial view of the semigroups.

A particularly simple family emerged and it is the main subject of the present paper. 
\medskip

\commentgray{rever}
\subsection*{Structure of the paper}\label{subsec:structure-paper}

The structure of the paper briefly follows. Besides some motivation and introducing some terminology, to which the first two sections are dedicated, the paper has several other sections.

Section~\ref{sec:Sp} is the heart of the paper. In it we show the existence of numerical semigroups with arbitrary large negative Eliahou number. To be precise, for each even positive integer $p$ we give a numerical semigroup $S(p)$ whose Eliahou number is $\frac{p}{2}(1-\frac{p}{2})$. 

In Section~\ref{sec:Sptau} we slightly modify $S(p)$, by adding a kind of remainder $\tau$. The numerical semigroup obtained is denoted $S(p,\tau)$. This construction leads us to conclude that every integer is the Eliahou number of a numerical semigroup.

Then, in Section~\ref{sec:Sijptau}, we give an infinite family of numerical semigroups $(S^{(i,j)}(p,\tau))_{i,j\in \mathbb N}$ whose Eliahou number is the same as the Eliahou number of $S(p,\tau)$. We thus conclude that for each given integer there are infinitely many numerical semigroups whose Eliahou number is that integer.
We prove that for any even positive integer $p$ and non negative integers $i,j$ and $\tau$, the semigroup $S^{(i,j)}(p,\tau)$ satisfies Wilf's conjecture.

In an appendix section we discuss a problem concerning the minimum possible genus of a numerical semigroup having a given Eliahou number.

In another appendix section we give a variety of examples, which essentially may be seen as a source of counter-examples. We explicitly state some remarks that are merely counter-examples to some questions that could be seen as natural. On the other hand, the examples given may help in the process of finding right questions to work on.

Throughout the paper we derive some interesting consequences of our results. As an example, we refer that for given integers $n$ and $N$, there are infinitely many numerical semigroups~$S$ such that $\eliahouoper(S)=n$ and $\wilfoper(S)>N$, as stated in Corollary~\ref{cor:large-Wilf}.


\section{Distinguished numbers, a convenient partition and figures}\label{sec:numbers-figures}

This section is dedicated to the introduction of some terminology and to fix some notation. Most of it is borrowed from Eliahou's paper~\cite{eliahou}. From the same paper we borrow a convenient partition of the integers.
For commonly used terminology and well known concepts we refer to a book of Rosales and García-Sánchez~\cite{NSbook}.
\subsection{Terminology and notation}\label{subsec:terminology}
Let~$S$ be a numerical semigroup.

The minimal generators of~$S$ are also known as \emph{primitive elements} of~$S$. The set of primitive elements of~$S$ is denoted $\primitivesoper(S)$. When~$S$ is understood, we usually simplify the notation and write simply $\primitivesoper$ instead of $\primitivesoper(S)$. This kind of simplification in the notation used is done for all the other invariants whose notation we now introduce.

The notation $\lvert X\rvert$ is used to denote the number of elements of a set $X$.
For instance, the cardinality of $\primitivesoper(S)$ (usually called the \emph{embedding dimension} of~$S$) is denoted $\lvert \primitivesoper(S)\rvert $ (or simply~$\lvert\primitivesoper\rvert$). 

The \emph{conductor} of~$S$ is the smallest integer in~$S$ from which all the larger integers belong to~$S$. It is denoted $\conductoroper(S)$ (or simply $\conductoroper$).

The \emph{multiplicity} of~$S$ is the least positive element of~$S$ and is denoted $\multiplicityoper(S)$ (or simply $\multiplicityoper$).

The set of \emph{left elements} of~$S$ consists of the elements of~$S$ that are smaller than $\conductoroper(S)$. It is denoted $\leftsoper(S)$ (or simply $\leftsoper$).
\commentgray{We observe that the concept of \emph{left elements} (introduced by Eliahou) differs from the concept of \emph{small elements}, which in addition contains the conductor, used in the numericalsgps package: the second one determines the numerical semigroup while the first does not.} 

The notation $\langle m,g_1,g_2,\ldots,g_r\rangle_{t}$ is used to represent the smallest numerical semigroup that contains $\{m,g_1,g_2,\ldots,g_r\}$ and all the integers greater than or equal to~$t$. (Here $t$ is not necessarily the conductor of the semigroup.)

The notation used for an interval of integers such as $\{t\in\mathbb{Z}\mid a\le t <b\}$ is $\left[a,b\right[$ (occasionally, $\left[a,b\right[\cap \mathbb{Z}$ is used). The set of non-negative integers is denoted by $\mathbb N$. 
\subsection{A convenient partition and figures}\label{subsec:partition-and-figures}
Let $\qoper(S)=\lceil \conductoroper(S)/\multiplicityoper(S)\rceil$ be the smallest integer greater than or equal to $\conductoroper(S)/\multiplicityoper(S)$. This number will frequently be called the \emph{$\qoper$-number} of~$S$ and denoted  simply by $\qoper$.
The interval of integers starting in $\conductoroper$ (the conductor) and having $\multiplicityoper$ (the multiplicity) elements is denoted $I_{\qoper}$, that is, 
$$I_{\qoper} = \left[\conductoroper, \conductoroper + \multiplicityoper\right[\cap \mathbb{Z}=\{z\in\mathbb{Z}\mid \conductoroper\le z < \conductoroper+\multiplicityoper\}.$$
For $j\in\mathbb{Z}$, denote $I_j$ the translate of $I_{\qoper}$ by $\left( j − \qoper \right)\multiplicityoper$, that is, $$I_j = I_{\qoper} + \left( j − \qoper \right)\multiplicityoper.$$

The sets $I_j$, with $j\in\mathbb{Z}$, form a partition of the integers that will be used throughout the paper.
Several figures will be presented to give pictorial views of numerical semigroups. Each picture in this paper (aiming to represent a numerical semigroup) consists of a rectangular $(\qoper+1)\times \multiplicityoper$-table such that the $j$-th row ($0\le j\le \qoper$) is the set $I_j$, and the entries corresponding to elements of the semigroup are somehow highlighted. 

We refer to row $i$ in the table (equivalently to the set $I_i$) as the \emph{$i$-th level}. When referring to columns, we keep using the term \emph{column}.

\begin{example}\label{ex:gens-11-13-21-62}
Figure~\ref{fig:gens-11-13-21-62} is a pictorial representation of the numerical semigroup $\langle 11,13,21,62 \rangle$. The elements of the semigroup are highlighted and, among them, the primitive elements and the conductor are emphasized. When an element is to be highlighted for more than one reason, gradient colours are used. 
\begin{figure}
\begin{center}
\includegraphics[scale=0.8]{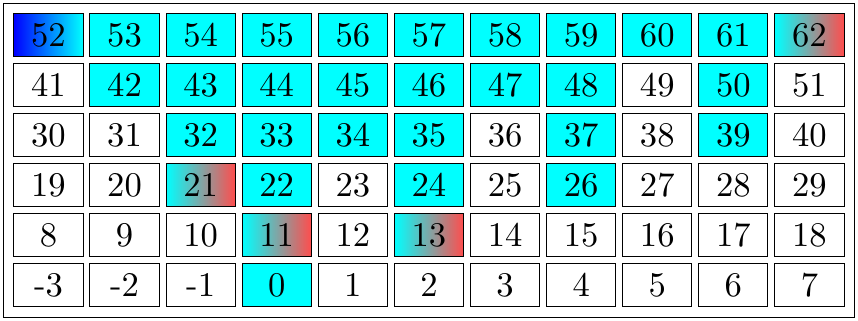}
\end{center}
\caption{Pictorial representation of the numerical semigroup $\langle 11,13,21,62 \rangle$. \label{fig:gens-11-13-21-62}}
\end{figure} 
\end{example}

We need some more notation:  $$\rhooper(S)=\qoper(S)\cdot \multiplicityoper(S)-\conductoroper(S).$$
As usual, most of the times we will simply use $\rho$, instead of $\rho(S)$.

For the semigroup considered in Example~\ref{ex:gens-11-13-21-62}, we have $\rho=3$, which is the number of columns to the left of $0$.
\subsection{Eliahou and Wilf's numbers}\label{subsec:wilf-eliahou-numbers}

To the numerical semigroup~$S$ one can associate the following number 
$$\wilfoper(S) = \lvert \primitivesoper \rvert\lvert \leftsoper \rvert-c,$$
which is called the \emph{Wilf number of~$S$}.
\smallskip

Wilf's conjecture~\cite{wilf78} may be stated as follows.
\begin{conjecture}[Wilf, 1978]\label{conj:wilf}
  Let~$S$ be a numerical semigroup. Then $\wilfoper(S)\ge 0$.
\end{conjecture}

The non primitive elements of $S$ are called \emph{decomposable}. The decomposable elements at level $\qoper$, $I_{\qoper}\setminus P$, are of particular importance and deserve the introduction of a notation. We write $\Dqoper$ for $I_{\qoper}\setminus P$.

We refer to the number 
$$\eliahouoper(S) = \lvert P\cap L\rvert\lvert \leftsoper \rvert - {\qoper} \lvert \Dqoper\rvert +\rho$$
as the \emph{Eliahou number of~$S$}.

Eliahou~\cite{eliahou} proved the following result, which relates $\eliahouoper(S)$ to the conjecture of Wilf.
\begin{proposition}\label{prop:non-negative-eliahou-implies-wilf}
If $\eliahouoper(S)\ge 0$, then $\wilfoper(S)\ge 0$.
\end{proposition}

\subsection{Fromentin-Eliahou examples}\label{subsec:fromentin}
The \emph{genus} of a numerical semigroup~$S$ is $\lvert\mathbb{N}\setminus S\rvert$.
The only numerical semigroups with negative Eliahou number among the over $10^{13}$ numerical semigroups of genus up to $60$ have been discovered by Fromentin through an exhaustive search. These semigroups, which are listed in Eliahou's paper~\cite{eliahou}, are: 
$\langle14,22,23\rangle_{56}$, $\langle16,25,26\rangle_{64}$, $\langle17,26,28\rangle_{68}$, $\langle17,27,28\rangle_{68}$ and $\langle18,28,29\rangle_{72}$. Figure~\ref{fig:fromentin-sgps} represents all of them.
Three of these numerical semigroups appear in the families we will be considering. Using names to be introduced, those semigroups are:  $S(4)=\langle14,22,23\rangle_{56}$, $S^{(0,1)}(4,0)=\langle16,25,26\rangle_{64}$ and $S^{(0,2)}(4,0)=\langle18,28,29\rangle_{72}$.
\begin{figure}
\includegraphics[scale=0.8]{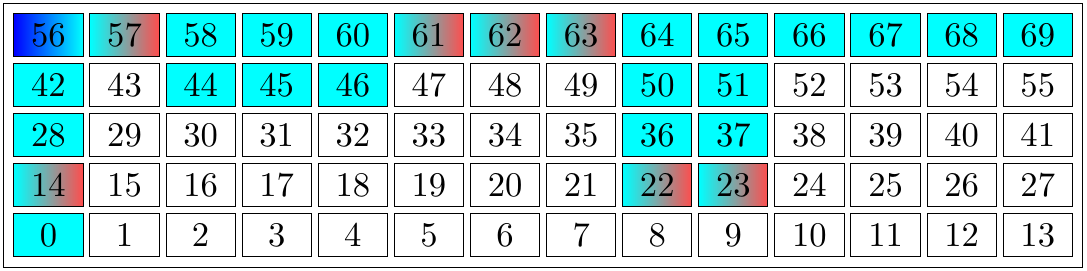}
\includegraphics[scale=0.8]{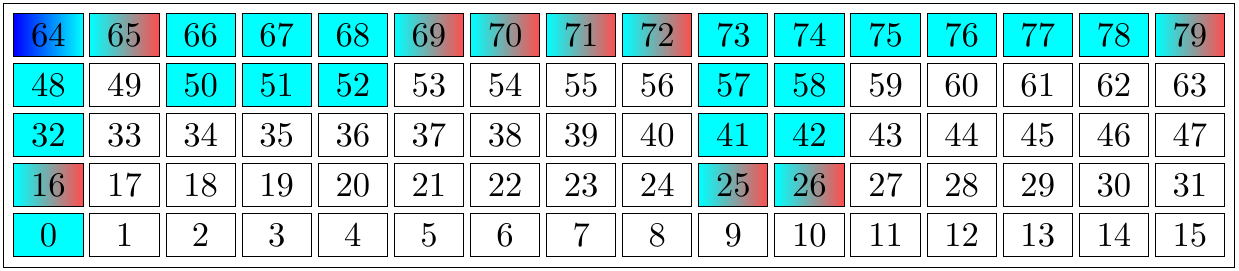}
\includegraphics[scale=0.8]{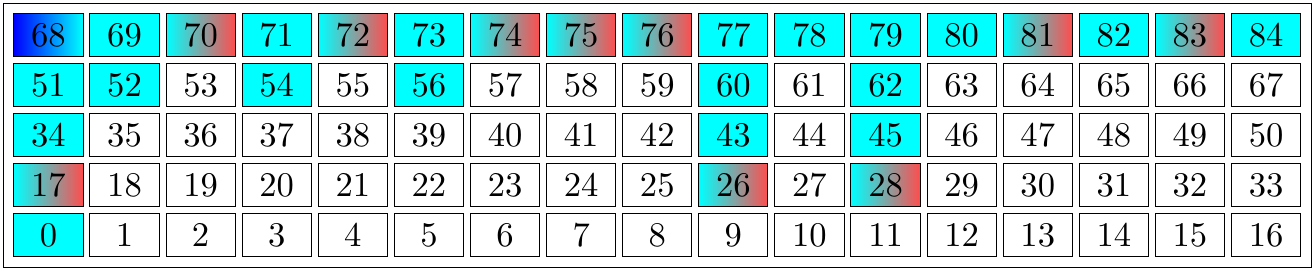}
\includegraphics[scale=0.8]{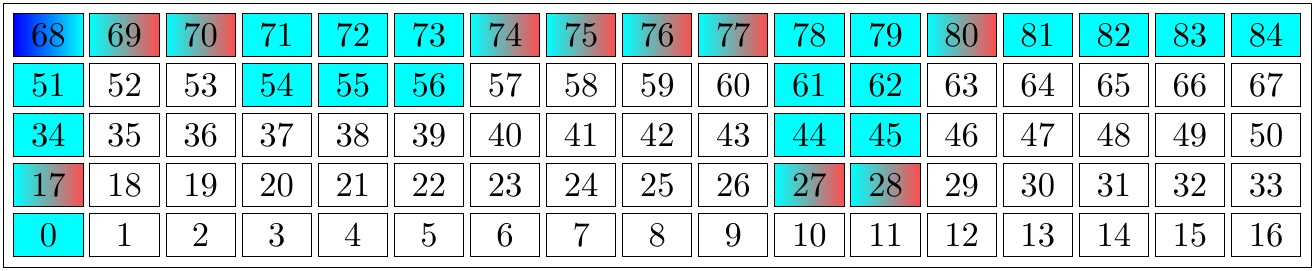}
\includegraphics[scale=0.8]{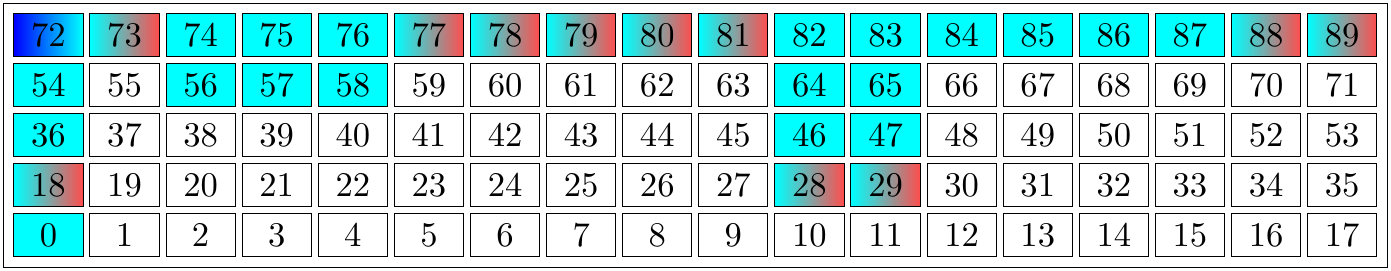}
\caption{Pictorial representation of the Fromentin-Eliahou examples. \label{fig:fromentin-sgps}}
\end{figure} 


\section{The Eliahou number can be arbitrarily large negative}\label{sec:Sp}
Throughout this section (in fact, through all the paper), $p$ is an \emph{even} positive integer. We remark that $p$ is related to the $\qoper$-number of the semigroups to be constructed and that we do not have analogous constructions for odd~$p$'s.

We associate to $p$ a numerical semigroup, $S(p)$, which has $3$ primitive left elements. The main aim of the section is to prove Theorem~\ref{th:wo_Sp}, which gives a formula for $\eliahouoper(S(p))$.

In Section~\ref{subsec:table-Sp} we collect in a table some numbers related to $S(p)$.
The last subsection gives an alternative (and much more visual) way to look at the numerical semigroups $S(p)$. It is very useful for modifications that lead to the construction of other families of semigroups to be considered in subsequent sections. 
\subsection{Defining the numerical semigroup $S(p)$}\label{subsec:defining-Sp}

To an even positive integer $p$, we associate a numerical semigroup, $S(p)$. To this end, we consider integers $\mu(p)$ and $\gamma(p)$, depending on $p$, defined as follows:
\begin{align}
\mu(p) &= \frac{p^2}{4}+2p+2= \frac{p}{2}\left(\frac{p}{2} +4\right)+2;\nonumber\\
\gamma(p) &= 2\mu(p) - \left(\frac{p}{2} +4\right).\nonumber
\end{align}
These integers will be used throughout the paper. We will generally write~$\mu$ instead of~$\mu(p)$ and~$\gamma$ instead of~$\gamma(p)$, since there will be no risk of confusion.

The smallest numerical semigroup containing $\{\mu,\gamma,\gamma+1\}$ and all the integers greater than or equal to~$p\mu$ plays a fundamental role in this paper. We use the notation:
$$S(p)=\langle \mu,\gamma,\gamma+1\rangle_{p\mu}.$$

  The interval $\left[ 0,(p+1)\mu \right[$ contains all the left elements and all the primitive elements of the numerical semigroup $S(p)$; it contains the important part of the semigroup.

One of the facts that we shall prove is that $\conductoroper(S(p))=p\mu$ (Corollary~\ref{cor:conductor-Sp}). Thus, the pictorial representation of the semigroup $S(p)$ consists of the integers in the interval $\left[ 0,(p+1)\mu \right[$, which are disposed in a rectangular table as described in Page~\pageref{subsec:partition-and-figures}.
\medskip

The semigroup $S(2)=\langle 7,9,10\rangle_{14}$ is pictorially represented in Figure~\ref{fig:2-0-1}.
\begin{figure}
\begin{center}
\includegraphics[scale=0.75]{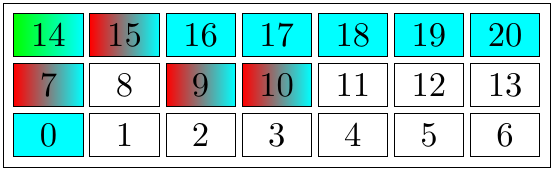}
\end{center}
\caption{The numerical semigroup $S(2)$. \label{fig:2-0-1}}
\end{figure} 

A pictorial representation of $S(4)=\langle14,22,23\rangle_{56}$ was already given in Figure~\ref{fig:fromentin-sgps}.

Figure~\ref{fig:Sp-shapes} (obtained by taking $p=10$) is intended to illustrate the general shape of a semigroup of the form $S(p)$ and may be useful to follow the proofs in the present section. 
\begin{figure}
\includegraphics[width=\textwidth]{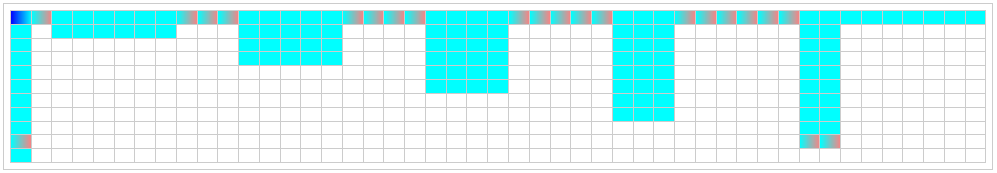}
\caption{Shape of a numerical semigroup of the form $S(p)$. \label{fig:Sp-shapes}}
\end{figure} 
Similar images are given along the paper. The reader may want to take a quick look to those of Figure~\ref{fig:Sp-shadowedshapes}, given in Section~\ref{subsec:visual-Sp}. 

\subsection{Some relations between multiples of $\mu$ and $\gamma$}\label{subsec:pic-rep-Sprel-mu-gamma}
The following technical results give simple relations between multiples of the integers $\mu$ and $\gamma$.
\begin{lemma}\label{lemma:relationgammaMu}
 $\gamma>2\left(1-\frac{1}{p}\right)\mu$.
\end{lemma}
\begin{proof}
We can write the following inequalities, which are clear.
\begin{align} 
\gamma &= 2\mu - \left(\frac{p}{2} +4\right)=2\mu - \frac{2}{p}\cdot\frac{p}{2}\left(\frac{p}{2} +4\right)\nonumber\\
&>2\mu - \frac{2}{p}\left[\frac{p}{2}\left(\frac{p}{2} +4\right)+2\right]=2\left(1-\frac{1}{p}\right)\mu.\nonumber \qedhere
\end{align}
\end{proof}
\begin{lemma}\label{lemma:compareGammaMu}
Let $r$ be an integer such that $0\le r\le \frac{p}{2}$. Then  $\left(\frac{p}{2}+r\right)\gamma > \left(p+2r-2\right)\mu$.
\end{lemma}
\begin{proof}
Using Lemma~\ref{lemma:relationgammaMu}, we can write, for any integer $r$,  $$\left(\frac{p}{2}+r\right)\gamma > 2\left(\frac{p}{2}+r\right)\left(1-\frac{1}{p}\right)\mu = \left(p+2r-1-\frac{2r}{p}\right)\mu.$$

Since we are assuming that $0\le r\le \frac{p}{2}$, we have that $0\le \frac{2r}{p} \le 1$, and therefore $$p+2r-1-\frac{2r}{p}\ge p+2r-2,$$ which completes the proof.
\end{proof}
Of special interest are the cases $r=1$ and $r=2$. These are the subject of the following corollary.
\begin{corollary}\label{cor:compareGammaMu} 
The following inequalities hold.
\begin{align}
&\left(\frac{p}{2}+1\right)\gamma > p\mu;\label{line1:cor:compareGammaMu}\\
&\left(\frac{p}{2}+2\right)\gamma > (p+1)\mu.\label{line2:cor:compareGammaMu}
\end{align}
\end{corollary}
\begin{proof}
Inequality~(\ref{line1:cor:compareGammaMu}) is an immediate consequence of Lemma~\ref{lemma:compareGammaMu}. If $p\ge 4$, then the same happens for Inequality~(\ref{line2:cor:compareGammaMu}).
For $p=2$, one can check that Inequality~(\ref{line2:cor:compareGammaMu}) also holds. In fact, as $\gamma(2)=9$ and $\mu(2)=7$, we get $27$ for the left side and $21$ for the right side.
\end{proof}

\subsection{The subsemigroup $\langle \gamma,\gamma+1\rangle$}\label{subsec:gammagammap1}
We take a close look at the subsemigroup $\langle \gamma,\gamma+1\rangle$ of $S(p)$ (which is a numerical subsemigroup, since $\gamma$ and $\gamma+1$ are coprime).
We start by stating the following observation.
\begin{lemma}\label{lemma:elts-gammagammap1}
$\langle \gamma,\gamma+1\rangle = \left\{i\gamma +j\mid i,j\in \mathbb{N} \mbox{ and } 0\le j\le i\right\}.$
\end{lemma}
\begin{proof}
As $0\le j\le i$, we have that $i\gamma +j =(i-j)\gamma+j(\gamma+1)\in \langle \gamma,\gamma+1\rangle$ and therefore we conclude that the rightmost set is contained in the leftmost one.

Observing that, for non-negative integers $x$ and $y$, we have $x\gamma+y(\gamma+1)=(x+y)\gamma+y$ and $y\le x+y$, the reverse inclusion follows.
\end{proof}

 As a consequence, we can write:
 \begin{align}
   \left[ 0,(p+1)\mu \right[\cap\langle \gamma,\gamma+1\rangle & =\left[ 0,(p+1)\mu \right[\cap\left\{i\gamma +j\mid i,j\in \mathbb{N} \mbox{ and } 0\le j\le i\right\}.\nonumber 
 \end{align}
 
Now, using Corollary~\ref{cor:compareGammaMu}, which ensures that, given non-negative integers $i$ and $j$, if $i\ge\left(p/2+1\right)$, then $i\gamma+j\ge p\mu$ and also that if $i\ge\left(p/2+2\right)$, then $i\gamma+j\ge (p+1)\mu$, we get the following:
\begin{proposition}\label{prop:preliminary-for-left-elts}
Let $p$ be an even integer. Then
\begin{align}
\left[ 0,p\mu \right[\cap\langle \gamma,\gamma+1\rangle &= \left\{i\gamma +j\mid 0\le i\le \frac{p}{2}, 0\le j\le i\right\};\label{eq:line1:preliminary-for-left-elts}\\
\left[ 0,(p+1)\mu \right[\cap\langle \gamma,\gamma+1\rangle &= \left\{i\gamma +j\mid 0\le i\le \frac{p}{2}+1, 0\le j\le i\right\}\label{eq:line2:preliminary-for-left-elts}.
\end{align}
\end{proposition}

\begin{corollary}\label{cor:preliminary_for_decomposable}
$$\left[ p\mu,(p+1)\mu \right[\cap\langle \gamma,\gamma+1\rangle = \left\{\left(\frac{p}{2}+1\right)\gamma +j\mid 0\le j\le \frac{p}{2}+1\right\}.$$
\end{corollary}

Since we are interested in knowing which elements of $\langle \gamma,\gamma+1\rangle$ belong to $\left[ 0,(p+1)\mu \right[$ (which constitutes the important part of $S(p)$), we can, by Proposition~\ref{prop:preliminary-for-left-elts}, restrict our study to the elements $i\gamma +j$, with $0\le i\le \frac{p}{2}+1, 0\le j\le i$.

\subsection{The conductor of $S(p)$}\label{subsec:conductor-Sp}

Next we will see how the elements of $\langle \gamma,\gamma+1\rangle$ are distributed through the various levels. As a consequence, we show that the conductor of $S(p)$ is $p\mu$.
Regarding the pictorial representation we adopted, the elements smaller than $p\mu$ are in odd levels and do not contain any element in the $\frac{p}{2}+2$ rightmost columns, as Proposition~\ref{prop:preliminary_for_left_elts_2} shows.
On the other hand, the elements at level $p$ are at the $\frac{p}{2}+2$ rightmost columns, as stated by Proposition~\ref{prop:preliminary_for_level_p}.

\begin{proposition}\label{prop:preliminary_for_left_elts_2}
Let $i$ and $j$ be such that $1\le i\le \frac{p}{2}, 0\le j\le i$ and let $i\gamma +j\in \langle \gamma,\gamma+1\rangle$. Then 
$$i\gamma +j\in \left[(2i-1)\mu,2i\mu-\left(\frac{p}{2}+3\right)\right].$$
\end{proposition}
\begin{proof}
 We have to prove the following inequalities:
\begin{align}
&(2i-1)\mu\le i\gamma; \label{eq:line1:preliminary_for_left_elts_2}\\
&i\gamma+i\le 2i\mu-\left(\frac{p}{2}+3\right).\label{eq:line2:preliminary_for_left_elts_2}
\end{align}
To prove (\ref{eq:line1:preliminary_for_left_elts_2}) we start by writing a sequence of equivalences:
$$(2i-1)\mu\le i\left(2\mu - \left(\frac{p}{2} +4\right)\right) \Longleftrightarrow -\mu\le -i\left(\frac{p}{2} +4\right) \Longleftrightarrow \mu\ge i\left(\frac{p}{2} +4\right).$$ 
The last inequality holds, since $\mu =\frac{p}{2}\left(\frac{p}{2}+4\right) +2 > \frac{p}{2}\left(\frac{p}{2}+4\right)$.
\smallskip

\noindent
To prove (\ref{eq:line2:preliminary_for_left_elts_2}), we also write a sequence of equivalences:
\begin{align}
 i\left(2\mu - \left(\frac{p}{2} +4\right)\right)+i\le 2i\mu-\left(\frac{p}{2}+3\right)  &\Longleftrightarrow -i\left(\frac{p}{2} +4\right)+i\le -\left(\frac{p}{2}+3\right)\nonumber\\
  &\Longleftrightarrow i\left(\left(\frac{p}{2} +4\right)-1\right)\ge \left(\frac{p}{2}+3\right).\nonumber
\end{align}
The last inequality is obvious, since $i\ge 1$.
\end{proof}

\begin{proposition}\label{prop:preliminary_for_level_p}
Let $j$ be such that $0\le j\le \frac{p}{2}+1$. Then 
$$\left(\frac{p}{2}+1\right)\gamma +j\in \left[(p+1)\mu-\left(\frac{p}{2}+2\right),(p+1)\mu\right[.$$
\end{proposition}
\begin{proof}
We will prove the following equalities:
\begin{align}
&\left(\frac{p}{2}+1\right)\gamma = (p+1)\mu - \frac{p}{2}-2; \label{eq:line1:preliminary_for_level_q}\\
&\left(\frac{p}{2}+1\right)\gamma + \frac{p}{2}+1= (p+1)\mu -1. \label{eq:line2:preliminary_for_level_q}
\end{align}
Notice that, as $\left(\frac{p}{2}+1\right)\gamma\le \left(\frac{p}{2}+1\right)\gamma +j\le\left(\frac{p}{2}+1\right)\gamma + \frac{p}{2}+1$, the result follows.
\medskip

\noindent
We have the following sequence of equalities, which proves~(\ref{eq:line1:preliminary_for_level_q}): 
\begin{align}
\left(\frac{p}{2}+1\right)\gamma 
&=\left(\frac{p}{2}+1\right)\left(2\mu - \left(\frac{p}{2} +4\right)\right)  = (p+2)\mu - \left(\frac{p}{2}+1\right)\left(\frac{p}{2} +4\right) \nonumber\\
&= (p+2)\mu - \left(\frac{p^2}{4}+5\frac{p}{2}+4\right) = (p+1)\mu +\frac{p^2}{4}+2p+2 - \left(\frac{p^2}{4}+5\frac{p}{2}+4\right)\nonumber \\
&= (p+1)\mu - \frac{p}{2}-2.\nonumber
\end{align}
Next we prove~(\ref{eq:line2:preliminary_for_level_q}):
\begin{align}
\left(\frac{p}{2}+1\right)\gamma + \frac{p}{2}+1 
&= \left(\frac{p}{2}+1\right)(\gamma+1) = \left(\frac{p}{2}+1\right)\left(2\mu - \left(\frac{p}{2} +4\right)+1\right)\nonumber\\
&= \left(\frac{p}{2}+1\right)\left(2\mu - \left(\frac{p}{2} +3\right)\right)= (p+2)\mu - \left(\frac{p}{2}+1\right)\left(\frac{p}{2} +3\right)\nonumber\\
&= (p+2)\mu - \left(\frac{p^2}{4}+2p+3\right) = (p+1)\mu -1.\nonumber \qedhere
\end{align}
\end{proof}

The announced result is a consequence of Propositions~\ref{prop:preliminary-for-left-elts} and~\ref{prop:preliminary_for_left_elts_2}.
\begin{corollary}\label{cor:conductor-Sp}
The conductor of $S(p)$ is $p\mu$.
\end{corollary}
\begin{proof}
As $\frac{p}{2}\gamma+\frac{p}{2}=p\mu-\frac{p}{2}\left(\frac{p}{2} +3\right)<p\mu-1$, the maximum of $\left\{i\gamma +j\mid 0\le i\le \frac{p}{2}, 0\le j\le i\right\}$ is smaller than $p\mu-1$.  It follows from Proposition~\ref{prop:preliminary-for-left-elts}~(\ref{eq:line1:preliminary-for-left-elts}) that $p\mu-1$ does not belong to $\langle \gamma,\gamma+1\rangle$. 

Furthermore, it follows from Proposition~\ref{prop:preliminary_for_left_elts_2} that no element of $\langle\gamma,\gamma+1\rangle$ smaller than $p\mu$ is congruent to $-1$ modulo $\mu$. We have thus proved that the conductor of the numerical semigroup $\langle \mu,\gamma,\gamma+1\rangle$ is no smaller than $p\mu$. 

As, by definition, the conductor of $S(p)$ is no greater than $p\mu$, the result follows.
\end{proof}
As an immediate consequence, we get the number $\qoper(S(p))$.
\begin{corollary}\label{cor:q-number-Sp}
The $\qoper$-number of $S(p)$ is $p$.
\end{corollary}
Recall that $\rho(S(p)) =\qoper(S(p))\cdot\multiplicityoper(S(p))-\conductoroper(S(p))$. Thus we have:
\begin{corollary}\label{cor:rho-Sp}
 $\rho(S(p)) = 0$.
\end{corollary}

The following also immediate consequence reduces to the formula to the Eliahou number of $S(p)$ in terms of $p$.
The left primitives are precisely the distinct integers $\mu$, $\gamma$ and $\gamma+1$.
\begin{corollary}\label{cor:left-primitives-Sp}
  $\lvert \primitivesoper(S(p))\cap \leftsoper(S(p))\rvert = \lvert\left\{\mu,\gamma,\gamma+1\right\}\rvert = 3$.
\end{corollary}

\subsection{The left elements of $S(p)$}\label{subsec:left-elements-Sp}
In what follows, having in mind a pictorial representation of $S(p)$ (such as the one given by Figure~\ref{fig:Sp-shapes}) may help to understand the formal writing of the results.

Let $i$ and $j$ be integers such that $0\le i\le \frac{p}{2}, 0\le j\le i$.  We define:
$$C_{i,j}=\{x\mu+i\gamma +j\mid x\in \mathbb{N} \mbox{ and } x\mu+i\gamma +j <p\mu\}.$$

Pictorially, $C_{i,j}$ is a column with $i\gamma +j$ in its base and containing all the elements above it, until level~$p-1$. 

Recall that, by Lemma~\ref{lemma:elts-gammagammap1}, all the elements of $\langle \gamma,\gamma+1\rangle\subseteq S(p)$ of level less than $p$ can be written in the form $i\gamma +j$. It follows that the set of left elements of $S(p)$ is the union of the $C_{i,j}$. We state it as a remark.
\begin{remark}\label{rem:union-columns}
  $\leftsoper(S(p))=\bigcup_{0\le i\le \frac{p}{2}, 0\le j\le i}C_{i,j}$.
\end{remark}

\begin{lemma}\label{lemma:the-set-Cij}
The cardinality of $C_{0,0}=\{x\mu\mid 0\le x <p\}$ is $p$.\\
If $i>0$, then 
\begin{equation}\label{eq:Cij}
C_{i,j}=\{x\mu+i\gamma +j\mid  x\in \mathbb{N} \mbox{ and } x \le p-2i\}.
\end{equation}
Moreover, the cardinality of $C_{i,j}$ is $p-2i + 1$.
\end{lemma}
\begin{proof}
The assertions concerning $C_{0,0}$ are immediate.

To prove the equality in (\ref{eq:Cij}), it suffices to prove
\begin{equation}\label{eq:Cij-proof}
  x\mu+i\gamma +j <p\mu \Longleftrightarrow x \le p-2i.
\end{equation}

Assume that $x \le p-2i$. Then 
\begin{align}
 x\mu+i\gamma +j &\le (p-2i)\mu + i\gamma+j = (p-2i)\mu + 2i\mu -i \left(\frac{p}{2}+4\right)+j \nonumber\\
                 & =p\mu -i \left(\frac{p}{2}+4\right)+j < p,\mu\nonumber
\end{align}
which proves the implication from right to left. For the converse, suppose that $x\ge p -2i +1$.
Then 
\begin{align}
 x\mu+i\gamma +j & \ge \left(p -2i +1 \right)\mu + 2i\mu -i \left(\frac{p}{2}+4\right)+j=p\mu + \mu -i \left(\frac{p}{2}+4\right)+j\nonumber\\
                 &\ge p\mu + \mu-\frac{p}{2}\left(\frac{p}{2}+4\right)+j\ge p\mu.\nonumber
\end{align}

The conclusion about the cardinality of $C_{i,j}$ is immediate.
\end{proof}
The columns $C_{i,j}$ are pairwise disjoint, that is, either do not intersect or are equal.
\begin{lemma}\label{lemma:Lij-disjoint}
Let $i, i', j$ and $j'$ be integers such that $0\le i,i'\le \frac{p}{2}$, $0\le j\le i$ and $0\le j'\le i'$.
If $C_{i,j}\cap C_{i',j'}\ne \emptyset$, then $i'=i$ and $j'=j$.
\end{lemma}
\begin{proof}
Let $x\mu+i\gamma +j=x'\mu+i'\gamma +j'\in C_{i,j}\cap C_{i',j'}$, with $0\le i\le \frac{p}{2}, 0\le j\le i$, and $0\le i'\le \frac{p}{2}, 0\le j'\le i'$.
The following sequence of implications holds:
\begin{align}
 x\mu+i\gamma +j=x'\mu+i'\gamma +j'&\implies i\gamma +j\equiv i'\gamma +j'\pmod{\mu}\nonumber\\
 &\implies i(p/2+4)+j\equiv i'(p/2+4)+j'\pmod{\mu}\nonumber\\
 &\implies (i-i')(p/2+4)+(j-j')\equiv 0\pmod{\mu}.\nonumber
\end{align}

Suppose now that $(i-i')(p/2+4)+(j-j')\ne 0$. We may assume that $i(p/2+4)+j> i'(p/2+4)+j'$.
Since we are assuming that $1\le i\le p/2$ and $0\le j\le i$, we get that $$(p/2+4)\le i(p/2+4)+j\le p/2(p/2+4)+p/2=p^2/2+2p+p/2.$$ On the other hand, since $i'\ge 1$, we get that $i'(p/2+4)+j'\ge p/2+4$. It follows that $$i(p/2+4)+j - (i'(p/2+4)+j')\le p^2/2+2p+p/2 - (p/2+4) = p^2/2+2p-4<\mu,$$ which is a contradiction. Thus  $(i-i')(p/2+4)+(j-j')= 0$.
\smallskip

As $\lvert j-j'\rvert <p/2+4$, we have that $(i-i')(p/2+4)+(j-j')=0$ implies $j=j'$ and therefore $i=i'$. It follows that $x'=x$ as well.
\end{proof}
The following lemma, which counts the left elements of $S(p)$, is crucial. 

\begin{lemma}\label{lemma:number_left_elts} Let $\leftsoper$ be the set of left elements of~$S$. Then
$$\lvert \leftsoper \rvert = p+\sum_{i=1}^{p/2}\left[p-(2i-1)\right](i+1).$$
\end{lemma}
\begin{proof}
  By Remark~\ref{rem:union-columns}, $\lvert \leftsoper \rvert = \lvert\bigcup_{0\le i\le \frac{p}{2}, 0\le j\le i}C_{i,j}\rvert$. As the $C_{i.j}$ are disjoint (Lemma~\ref{lemma:Lij-disjoint}), we can write: 
$$\left\lvert\bigcup_{0\le i\le \frac{p}{2}, 0\le j\le i}C_{i,j}\right\rvert = \bigcup_{0\le i\le \frac{p}{2}, 0\le j\le i}\lvert C_{i,j}\rvert.$$ 
The result follows from Lemma~\ref{lemma:the-set-Cij}, which gives the cardinality of each of the $C_{ij}$.
\end{proof}

As a consequence, we get a formula for the number of left elements.

\begin{corollary}\label{cor:number_left_elts}
  The number $\leftsoper$ of left elements of $S(p)$ is given by:
$$\lvert L\rvert = \frac{p^3}{24}+\frac{3}{8}p^2 + \frac{13}{12}p.$$
\end{corollary}
\begin{proof}
  The proof uses Lemma~\ref{lemma:number_left_elts} and formulas for the sum of consecutive integers and the sum of squares of consecutive integers.

  \begin{align}
\lvert L\rvert &=p+\sum_{i=1}^{p/2}\left[p-(2i-1)\right](i+1)
               = p+\sum_{i=1}^{\frac{p}{2}}\left((p+1)+(p-1)i-2i^2\right)\nonumber \\
               &=p+(p+1)\frac{p}{2}+(p-1)\frac{\frac{p}{2}\left(\frac{p}{2}+1\right)}{2}-2\frac{\frac{p}{2}\left(\frac{p}{2}+1\right)(p+1)}{6}\nonumber \\
               &=p+\frac{p}{2}\left((p+1)+(p-1)\frac{\left(\frac{p}{2}+1\right)}{2}-\frac{\left(\frac{p}{2}+1\right)(p+1)}{3}\right)\nonumber\\
               &=p+\frac{p}{2}\left(p+1+\left(\frac{p}{2}+1\right)\left(\frac{(p-1)}{2}-\frac{(p+1)}{3}\right)\right)\nonumber \\
               &=p+\frac{p}{2}\left(p+1+\left(\frac{p}{2}+1\right)\left(\frac{p-5}{6}\right)\right)\nonumber \\
	       &=p+\frac{p}{2}\left(p+1+\frac{p^2}{12}-\frac{3}{12}p-\frac{5}{6}\right)\nonumber \\
               &=p+\frac{p}{2}\left(\frac{p^2}{12}+\frac{9}{12}p+\frac{1}{6}\right)
               =\frac{p^3}{24}+\frac{9}{24}p^2+\frac{13}{12}p   \nonumber \qedhere
  \end{align}
\end{proof}

\subsection{The Eliahou number of $S(p)$}\label{subsec:Eliahou-Sp}
The set of decomposable elements of $S(p)$ in level $\qoper$ is denoted $\Dqoper$. We aim to compute its cardinality.
Let us denote by $A$ the set of elements of the form $n+\mu$, where $n$ is an element at level $p-1$, and by $B$ the set of elements of level $p$ that belong to $\langle \gamma,\gamma+1\rangle$.

\begin{proposition}\label{prop:Dp}
 The following equality holds: $\Dqoper=A\cup B$. Moreover, the union is disjoint.
\end{proposition}
\begin{proof}
 That the equality holds is straightforward.
 
 That the union is disjoint follows from the fact that $A$ is contained in the the leftmost ${\mu-\left(\frac{p}{2}+3\right)}$ columns (as a consequence of  Proposition~\ref{prop:preliminary_for_left_elts_2}) and $B$ is contained in the $\frac{p}{2}+2$ rightmost columns (by Proposition~\ref{prop:preliminary_for_level_p}). 
\end{proof}

\begin{corollary}\label{cor:number_decomposable_elts}
  The number of decomposable elements of $S(p)$ in level $\qoper$ is given by:
$$\lvert \Dqoper\rvert=\frac{p^2}{8}+\frac{5}{4}p+3/2.$$
\end{corollary}
\begin{proof}
  The cardinality of $A$ is precisely the number of columns containing the left elements (here we make use of Lemma~\ref{lemma:Lij-disjoint}), and that number is $\sum_{i=0}^{\frac{p}{2}}(i+1)$, since for each $i$ there are $i+1$~$j$'s.
  From Corollary~\ref{cor:preliminary_for_decomposable} the cardinality of $B$ is $\left(\frac{p}{2}+1\right)+1$.
  Now, using Proposition~\ref{prop:Dp} we get $$\lvert \Dqoper\rvert=1+\ldots+\left(\frac{p}{2}+1\right) + \left(\left(\frac{p}{2}+1\right)+1\right),$$
  which is the sum of the first $\frac{p}{2}+2$ positive integers. Therefore $$\lvert \Dqoper\rvert=\left(\frac{p}{2}+2\right)\left(\frac{p}{2}+3\right)/2.$$ The result follows.
\end{proof}

We may now state the main result of this section and probably of the whole paper.

\begin{theorem}\label{th:wo_Sp}
  Let $p$ be an even integer. The Eliahou number of $S(p)$ is given by:
$$\eliahouoper (S(p))=\frac{p}{4}\left(1-\frac{p}{2}\right).$$
\end{theorem}
\begin{proof}
From Corollary~\ref{cor:number_left_elts} it follows that 
\begin{equation}\label{eq:3l}
3\lvert L\rvert=\frac{p^3}{8}+\frac{9}{8}p^2+\frac{13}{4}p. 
\end{equation}

From Corollary~\ref{cor:number_decomposable_elts} we get that  
\begin{equation}\label{eq:dp}
p\lvert \Dqoper\rvert=p\left(\frac{p^2}{8}+\frac{5}{4}p+3\right)= \frac{p^3}{8}+\frac{5}{4}p^2+3p.
\end{equation}

By Corollary~\ref{cor:rho-Sp}, which states that $\rho(S(p))=0$, and using the fact that $S(p)$ has $3$ left primitives (Corollary~\ref{cor:left-primitives-Sp}), we can write:
$$\eliahouoper (S(p))=3\lvert L\rvert - p\lvert \Dqoper\rvert.$$

Now we just have to use (\ref{eq:3l}) and (\ref{eq:dp}) to complete the proof:
  \begin{align}
\eliahouoper (S(p))&=\left(\frac{p^3}{8}+\frac{9}{8}p^2+\frac{13}{4}p\right)-\left(\frac{p^3}{8}+\frac{5}{4}p^2+3p\right)= -\frac{p^2}{8}+\frac{p}{4}.\nonumber\qedhere
  \end{align}
\end{proof}

\begin{corollary}\label{cor:sgps_with_arbitrarily_large_negative_W0}
There exist numerical semigroups with arbitrarily large negative Eliahou number.
\end{corollary}

\subsection{The $S(p)$ satisfy Wilf's conjecture}\label{subsec:wilf-Sp}
As a consequence of Corollary~\ref{cor:number_decomposable_elts} we get the number of primitive elements of~$S(p)$.
\begin{corollary}\label{cor:embedding_dimension}
$\lvert\primitivesoper(S(p))\rvert = \frac{p^2}{8}+\frac{3}{4}p+2$.
\end{corollary}
\begin{proof}
The number of primitive elements of $S(p)$ is $3$ (the ones smaller than the conductor) plus the non-decomposables of $\left[p\mu,p\mu+\mu\right[$. Therefore, by using Corollary~\ref{cor:number_decomposable_elts}, we get that
\begin{align}
\primitivesoper(S(p))&=3 + \mu -\lvert \Dqoper\rvert= 3+ \frac{p^2}{4}+2p+2 - \left(\frac{p^2}{8}+\frac{5}{4}p+3\right) = \frac{p^2}{8}+\frac{3}{4}p+2.\nonumber\qedhere
  \end{align}
\end{proof}

The next proposition shows that Wilf's conjecture holds for all numerical semigroups~$S(p)$.

\begin{proposition}\label{prop:wilf_holds_for_all_p}
  Let $p$ be an even positive integer. Then $$\wilfoper (S(p)) =\frac{p^5}{192}+\frac{5p^4}{64}+\frac{p^3}{4}-\frac{7p^2}{16}+\frac{p}{6}.$$ 
  In particular, $\wilfoper (S(p)) > 0$. 
\end{proposition}
\begin{proof}
We use Corollary~\ref{cor:embedding_dimension} for $\lvert P\rvert$, Corollary~\ref{cor:number_left_elts} for $\lvert L\rvert$ and that the conductor of $S(p)$ is $p\mu$ (Corollary~\ref{cor:conductor-Sp}).

\begin{align}
&\wilfoper(S(p)) = \left(\frac{p^2}{8}+\frac{3}{4}p+2\right)\left(\frac{p^3}{24}+\frac{3}{8}p^2 + \frac{13}{12}p\right)-p\mu\nonumber\\
&\quad = \frac{p^5}{8\cdot 24}+\left(\frac{3}{8\cdot 8}+ \frac{3}{4\cdot 24}\right)p^4+\left( \frac{13}{8\cdot 12}+\frac{3\cdot 3}{4\cdot 8}+\frac{2}{24}\right)p^3+\left(\frac{3\cdot 13}{4\cdot 12}+\frac{6}{8} \right)p^2+\left( \frac{2\cdot 13}{12}\right)p-\mu\nonumber\\
&\quad =\frac{p^5}{3\cdot 2^6}+\left(\frac{15}{3\cdot 2^6}\right)p^4+\left( \frac{13+27+8}{3\cdot 2^5}\right)p^3+\left(\frac{25}{2^4 }\right)p^2+\left( \frac{2\cdot 13}{12}\right)p-\left(\frac{p^2}{4}+2p+2\right)\nonumber\\
&\quad =\frac{p^5}{192}+\frac{5p^4}{64}+\frac{p^3}{4}-\frac{7p^2}{16}+\frac{p}{6}.\nonumber
\end{align}

As for $p = 2$ we have that $\frac{p^5}{192}+\frac{5p^4}{64}+\frac{p^3}{4}-\frac{7p^2}{16}+\frac{p}{6}=2$ and $\frac{p^5}{192}+\frac{5p^4}{64}+\frac{p^3}{4}-\frac{7p^2}{16}+\frac{p}{6}$ is increasing, we get that $\wilfoper(S(p))>0$.
\end{proof}
The following remark  was somehow unexpected.
\begin{remark}\label{rem:unexpected}
 When $p$ grows, $\eliahouoper(S(p))$ becomes large negative while $\wilfoper(S(p))$ becomes large positive.
\end{remark}

\subsection{A table with numbers related to $S(p)$}\label{subsec:table-Sp}
We collect in Table~\ref{table:examples-Sp} some numbers related to $S(p)$.
Each row in the table contains several cells. The first one consists of the notation we use for the semigroup. The cells in the second and third columns contain respectively the Eliahou and the Wilf number. The values in the second and third columns illustrate Remark~\ref{rem:unexpected}. Then there is a cell containing the semigroup itself, showing the left minimal generators and the conductor. The cells in the last column contain the genus of the semigroups. 
The semigroups with large negative Eliahou numbers are deep in the semigroup tree. This point is further developed in Appendix~\ref{sec:remark-genus}.

\commentgray{GENUS?}

\begin{table}
$\begin{array}{|l|c|c|c|c|}
\hline
\text{Semigroup}&\eliahouoper&\wilfoper&\langle \text{ left generators }\rangle_c&\text{genus}\\ \hline\hline
S(2)&0&2&\langle 7,9,10\rangle_{14}&10\\
\hline
S(4)&-1&35&\langle 14,22,23\rangle_{56}&43\\
\hline
S(6)&-3&181&\langle 23,39,40\rangle_{138}&109\\
\hline
S(8)&-6&592&\langle 34,60,61\rangle_{272}&218\\
\hline
S(10)&-10&1510&\langle 47,85,86\rangle_{470}&380\\
\hline
S(12)&-15&3287&\langle 62,114,115\rangle_{744}&605\\
\hline
S(14)&-21&6405&\langle 79,147,148\rangle_{1106}&903\\
\hline
S(16)&-28&11496&\langle 98,184,185\rangle_{1568}&1284
\\\hline
S(18)&-36&19362&\langle 119,225,226\rangle_{2142}&1758
\\\hline
S(20)&-45&30995&\langle 142,270,271\rangle_{2840}&2335
\\\hline
\end{array}$
\caption{Numbers related to $S(p)$}
\label{table:examples-Sp}
\end{table}


\subsection{A visual way to look at $S(p)$}\label{subsec:visual-Sp}
Figure~\ref{fig:Sp-shadowedshapes} is intended to illustrate the general shape of a semigroup of the form $S(p)$. (One of the images was made taking $p=10$; the other $p=12$.) The light-grey tones are intended to better visualize several blocks of columns to be considered.

\begin{figure}
\includegraphics[width=\textwidth]{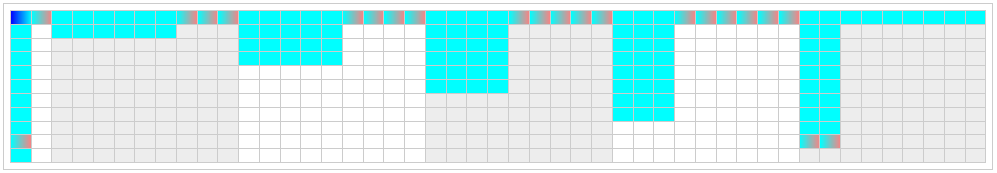}
\includegraphics[width=\textwidth]{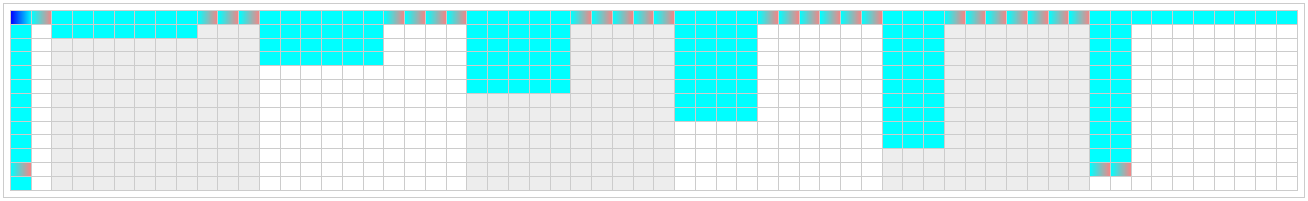}
\caption{Shapes of numerical semigroups of the form $S(p)$ with shadowed column blocks. \label{fig:Sp-shadowedshapes}}
\end{figure} 

One can split the tables in Figure~\ref{fig:Sp-shadowedshapes} into several blocks. We consider a block $C$ consisting of the $2$ leftmost columns; then we consider blocks $B_i$, $1\le i\le \frac{p}{2}$, consisting of $\frac{p}{2}+4$ consecutive columns. A block $B_i$ is to the left of a block $B_j$ if and only if $i<j$. We use a grey tone to visually separate blocks. Observe that the block $C$ together with the $p/2$ blocks $B_i$ partition the table. In fact, by construction, these blocks are pairwise disjoint. Furthermore, the blocks cover the table, since the number of columns covered is $\frac{p}{2}\left(\frac{p}{2}+4\right)+2=\mu$.
\smallskip

For any of the blocks $B$ above, denote by $B'$ the block obtained from $B$ by removing the uppermost row. The block $C'$ consists of a column of left elements and another that does not contain left elements. For each $i$ such that $1\le i\le \frac{p}{2}$, $B'_i$ consists of $\frac{p}{2}+2-i$ columns that contain some left element and $2+i$ that do not contain any ones. 

\section{Every integer is the Eliahou number of a numerical semigroup}\label{sec:Sptau}
Throughout this section, $p$ is an \emph{even} positive integer; $\mu$ and $\gamma$ are as in the previous section. Let, in addition, $\tau$ be a non-negative integer.

\subsection{Defining the numerical semigroup $S(p,\tau)$}\label{subsec:def-Sptau}
Consider the following integers:
\begin{itemize}
\item[] $m = \mu+\tau\frac{p}{2}=\frac{p}{2}\left(\frac{p}{2}+4+\tau\right) + 2 =\frac{p^2}{4}+2p+2+\tau\frac{p}{2}$;
\item[] $g = \gamma+\tau\left(p-1\right)= 2m-\left(\frac{p}{2}+4+\tau\right)=2m - (\tau+1)-\left(\frac{p}{2} +2+1\right)$;
\item[] $c = p\mu+\tau\left(\frac{p^2}{2}-1\right)=pm-\tau$.
\end{itemize}

Note that $c>0$. We define the semigroup $S(p,\tau)$ as being $\langle m,g,g+1\rangle_{c}$. Observe that $S(p)=S(p,0)$.
It is straightforward to observe that $c$ is the conductor of $S(p,\tau)$ (one can mimic the arguments that lead to Corollary~\ref{cor:conductor-Sp}). We register it as a proposition.
\begin{proposition}\label{prop:conductor_Sptau}
The conductor of $S(p,\tau)$ is $c=\conductoroper(S(p))+\tau\left(\frac{p^2}{2}-1\right) = pm-\tau$.
\end{proposition}
We can write the following easy consequences:
\begin{corollary}\label{cor:q-number-Spr}
$\qoper(S(p,\tau)) =\qoper(S(p))=p$.
\end{corollary}
\begin{corollary}\label{cor:rho-Sptau}
 $\rho(S(p,\tau)) = \tau$.
\end{corollary}
\begin{corollary}\label{cor:left-primitives-Sptau}
  $\lvert \primitivesoper(S(p,\tau))\cap \leftsoper(S(p,\tau))\rvert =\lvert \left\{ m,g,g+1\right\}\rvert= 3$.
\end{corollary}

\subsection{A visual way to look at $S(p,\tau)$}\label{subsec:visual-Sptau}

Pictorial representations of the semigroups $S(p,\tau)$ are obtained through slightly modifying the images of Figure~\ref{fig:Sp-shapes}, which represent semigroups $S(p)$. It is straightforward to observe that the images obtained as described below represent the semigroups $S(p,\tau)$.

Consider a non-negative integer $\tau$ and, for each of the images of Figure~\ref{fig:Sp-shadowedshapes} (Section~\ref{subsec:visual-Sp}), consider the image obtained by adding $\tau$ columns to the right of each of the blocks $B_i$, for $1\le i\le \frac{p}{2}-1$, and assume that it contains no left element of the semigroup obtained. Concerning the block $B_{\frac{p}{2}}$, instead of adding columns to the right, $\tau$ columns are put at the left of the table. Block $B$ is therefore split into two parts. Block $C$ now consists of the columns $\tau+1$ and $\tau+2$. 

The result of this modification is in Figure~\ref{fig:Sptau-shadowedshapes}. (The first image was made taking $p=10, \tau=3$; the second, taking $p=12,\tau=3$.)
\begin{figure}
\includegraphics[width=\textwidth]{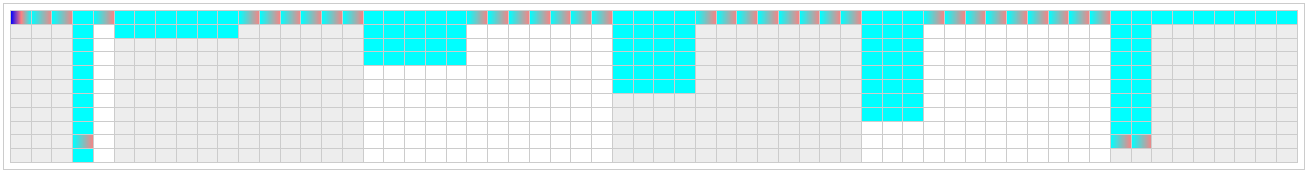}
\includegraphics[width=\textwidth]{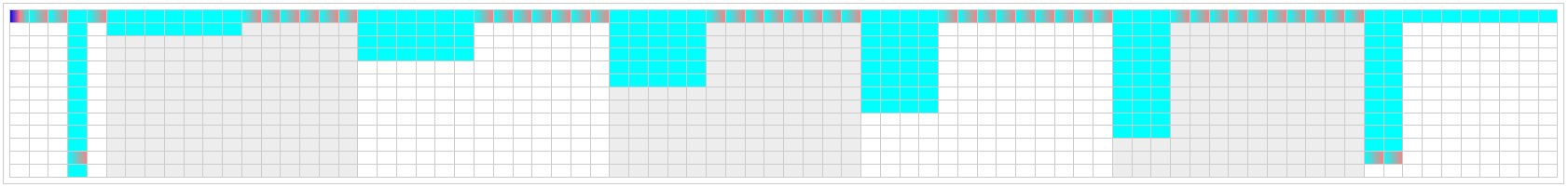}
\caption{Shapes of numerical semigroups of the form $S(p,\tau)$. \label{fig:Sptau-shadowedshapes}}
\end{figure}

When the integer $\tau$ is big, the shape of the representation obtained keeps the hight but becomes wider as (up to the scale) Figure~\ref{fig:Spbigtau} illustrates. 
\begin{figure}
\includegraphics[width=\textwidth]{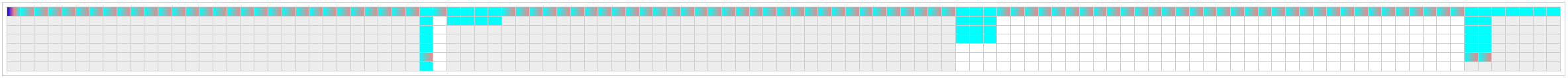}
\caption{Shape of a numerical semigroup of the form $S(p,\tau)$, with a big $\tau$. \label{fig:Spbigtau}}
\end{figure} 
\subsection{More on $S(p,\tau)$}\label{subsec:more-on-Sptau}
The following remarks can be proven in a straightforward way, as we did when obtaining the analogous results for $S(P)$, or relying on the images as the ones presented in Figures~\ref{fig:Sp-shadowedshapes} and~\ref{fig:Sptau-shadowedshapes}.
Note that each of the columns added leads to adding $p$ gaps and one new primitive element. 
In particular, the number of left elements and the number of decomposables in level $p$ remains unchanged. We register these facts as remarks.
\begin{remark}\label{rem:number_decomposables}
The semigroups $S(p)$ and $S(p,\tau)$ have the same number of decomposables in level $p$, that is, 
 $\lvert \Dqoper(S(p,\tau))\rvert=\lvert \Dqoper(S(p))\rvert$.
\end{remark}
\begin{remark}\label{rem:number_lefts-Sp-Sptau}
$\lvert\leftsoper(S(p,\tau))\rvert=\lvert\leftsoper(S(p))\rvert$.
\end{remark}

Our previous observations show that, excepting $\rho$ (Corollary~\ref{cor:rho-Sptau}), the parameters involved in the definition of the Eliahou number remain the same when replacing $S(p)$ by $S(p,\tau)$. Thus, we can use Theorem~\ref{th:wo_Sp} to obtain the following:
\begin{theorem}\label{th:formula_w0}
Let $p$ be an even positive integer and let $\tau$ be a non-negative integer. Then $$\eliahouoper (S(p,\tau))=\eliahouoper (S(p))+\rho(S(p,\tau)) = \frac{p}{4}\left(1-\frac{p}{2}\right)+\tau.$$
\end{theorem}

Let $n$ be an integer and take an even integer $p$ such that $\frac{p}{4}\left(1-\frac{p}{2}\right)\le n$. Now take $\tau = n- \frac{p}{4}\left(1-\frac{p}{2}\right)$, which is a non-negative integer. Then $\eliahouoper (S(p,\tau))=\frac{p}{4}\left(1-\frac{p}{2}\right)+\tau=\frac{p}{4}\left(1-\frac{p}{2}\right)+n- \frac{p}{4}\left(1-\frac{p}{2}\right)=n$. We have proved the following:
\begin{corollary}\label{cor:every_integer_is_Eliahou_number}
Every integer is the Eliahou number of some numerical semigroup.
\end{corollary}

One can be more precise by considering two different cases, depending on the sign of $n$:
\begin{corollary}\label{cor:every_positive_integer_is_Eliahou_number}
Let $n$ be a non-negative integer. Then $\eliahouoper (S(2,n))=n$.  
\end{corollary}
\begin{corollary}\label{cor:every_negative_integer_is_Eliahou_number}
Let $n$ be a negative integer and let $p$ be an even integer such that $p\ge 1+\sqrt{1-8n}$. Then $\eliahouoper \left(S(p,n-\frac{p}{4}\left(1-p/2\right))\right)=n$.  
\end{corollary}

\subsection{The $S(p,\tau)$ satisfy Wilf's conjecture}\label{subsec:wilf-Sptau}
It is clear that each of the columns added leads to adding gaps and a new primitive element. Thus, we can write the following remark.
\begin{remark}\label{rem:number_primitives}
$\lvert \primitivesoper(S(p,\tau))\rvert = \lvert \primitivesoper(S(p))\rvert + \tau\frac{p}{2}$.
\end{remark}
By using Corollary~\ref{cor:embedding_dimension}, we get
\begin{remark}\label{rem:number_primitives-Sptau}
$\lvert \primitivesoper(S(p,\tau))\rvert = \frac{p^2}{8}+\frac{3}{4}p+2 + \tau\frac{p}{2}=\frac{p^2}{8}+\frac{3+2\tau}{4}p+2 = \frac{1}{8}\left(p^2+(6+4\tau)p+16)\right)$.
\end{remark}

The following proposition gives, in particular, a formula (depending on $p$ and $\tau$) for the Wilf number of $S(p,\tau)$.
\begin{proposition}\label{prop:wilf_holds_for_all_p_rho}
Let $p$ be an even positive integer and let $\tau$ be a non-negative integer.
Then 
\begin{align}
\wilfoper (S(p,\tau))&= \wilfoper(S(p))   + \tau\cdot\left(\frac{p^4}{48}+\frac{3}{16}p^3 + \frac{1}{24}p^2 + 1\right)\nonumber
\end{align}
In particular, we get that $\wilfoper (S(p,\tau)) > 0$. 
\end{proposition}
\begin{proof}
Recall that $S(p)$ and $S(p,\tau)$ have the same number of left elements (Remark~\ref{rem:number_lefts-Sp-Sptau}), which is given by Corollary~\ref{cor:number_left_elts}. Using Remark~\ref{rem:number_primitives} and Proposition~\ref{prop:conductor_Sptau}, we can write the following sequence of equalities.
\begin{align}
\lvert \primitivesoper(S(p,\tau))\rvert\lvert\leftsoper\rvert-\conductoroper(S(p,\tau))  
&= \left(\lvert\primitivesoper(S(p))\rvert + \tau\frac{p}{2}\right)\lvert\leftsoper\rvert-\left(\conductoroper(S(p))+\tau\left(\frac{p^2}{2}-1\right)\right)\nonumber\\
&= \lvert\primitivesoper(S(p))\rvert  \lvert\leftsoper\rvert -\conductoroper(S(p))   + \tau\frac{p}{2}\lvert\leftsoper\rvert-\tau\left(\frac{p^2}{2}-1\right)\nonumber\\
&= \wilfoper(S(p))   + \tau\left(\frac{p}{2}\left(\lvert\leftsoper\rvert-p\right) + 1\right)\nonumber\\
&= \wilfoper(S(p))   + \tau\left(\frac{p}{2}\left(\frac{p^3}{24}+\frac{3}{8}p^2 + \frac{13}{12}p-p\right) + 1\right)\nonumber\\
&= \wilfoper(S(p))   + \tau\left(\frac{p^4}{48}+\frac{3}{16}p^3 + \frac{1}{24}p^2 + 1\right)\nonumber\qedhere
\end{align}
\end{proof}


\subsection{A table with numbers related to $S(p,\tau)$}\label{subsec:table-Sptau}
Analogously to what has been done in Section~\ref{subsec:table-Sp}, we collect in Table~\ref{table:examples-Sptau} some numbers related to $S(p,\tau)$.
\begin{table}
$\begin{array}{|l|c|c|c|c|}
\hline
\text{Semigroup}&\eliahouoper&\wilfoper&\langle \text{ left generators }\rangle_c&\text{genus}\\ \hline\hline
S(2,3)&3&11&\langle 10,12,13\rangle_{17}&13
\\\hline
S(4,3)&2&92&\langle 20,31,32\rangle_{77}&64
\\\hline
S(6,3)&0&391&\langle 32,54,55\rangle_{189}&160
\\\hline
S(8,3)&-3&1147&\langle 46,81,82\rangle_{365}&311
\\\hline
S(10,3)&-7&2713&\langle 62,112,113\rangle_{617}&527
\\\hline
S(12,3)&-12&5576&\langle 80,147,148\rangle_{957}&818
\\\hline
S(14,3)&-18&10377&\langle 100,186,187\rangle_{1397}&1194
\\\hline
S(16,3)&-25&17931&\langle 122,229,230\rangle_{1949}&1665
\\\hline
S(18,3)&-33&29247&\langle 146,276,277\rangle_{2625}&2241
\\\hline
S(20,3)&-42&45548&\langle 172,327,328\rangle_{3437}&2932
\\\hline
\end{array}$
\caption{$S(p,\tau)$ numbers}
\label{table:examples-Sptau}
\end{table}


\section{There are infinitely many numerical semigroups with a given Eliahou number}\label{sec:Sijptau}
Throughout this section, $p$ is an \emph{even} positive integer; $m$, $g$ , $c$ and $\tau$ are as in previous section. 
For each numerical semigroup $S(p,\tau)$ and each pair $(i,j)$ of non-negative integers, we define a numerical semigroup $S^{(i,j)}(p,\tau)$. We will prove that its Eliahou number is $\eliahouoper(S(p,\tau))$ (Theorem~\ref{th:same_w0-Sijptau}).

\subsection{Defining numerical semigroups $S^{(i,j)}(p,\tau)$}\label{sec:def-Sijptau}
Let $i$ and $j$ be non-negative integers and define $m^{(i,j)}$, $g^{(i,j)}$ and $c^{(i,j)}$ as follows. 
\begin{itemize}
\item[] $m^{(i,j)} = m+ j\frac{p}{2}$;
\item[] $g^{(i,j)} = g + j\cdot(p-1) + i\cdot m^{(i,j)}$;
\item[] $c^{(i,j)} = c + j\frac{p^2}{2} + i\left(\frac{p}{2}+1\right)m^{(i,j)}$.
\end{itemize}

Now, define: $$S^{(i,j)}(p,\tau)=\langle m^{(i,j)},g^{(i,j)},g^{(i,j)}+1\rangle_{c^{(i,j)}}.$$
The following remark is obvious and will be used without explicit mention.
\begin{remark}\label{rem:S00pr--Spr}
$S^{(0,0)}(p,\tau)=S(p,\tau)$.
\end{remark}

Before continuing, let us look at some pictures. Figure~\ref{fig:column-shades} illustrates blocks similar to those that we already encountered in Sections~\ref{subsec:visual-Sp} and~\ref{subsec:visual-Sptau}. From now on we refer to this kind of blocks as column-blocks. The second image is obtained from the first by adding one column to the right of each column-block. The first image represents $S(10,3)$, the second represents $S^{(0,1)}(10,3)$.

Recall that the construction of $S(p,\tau)$ from $S(p)$ was similar: it was based on on adding some columns to each column block. 

\begin{figure}
\includegraphics[width=\textwidth]{10-3-1-shape.pdf}
\includegraphics[width=\textwidth]{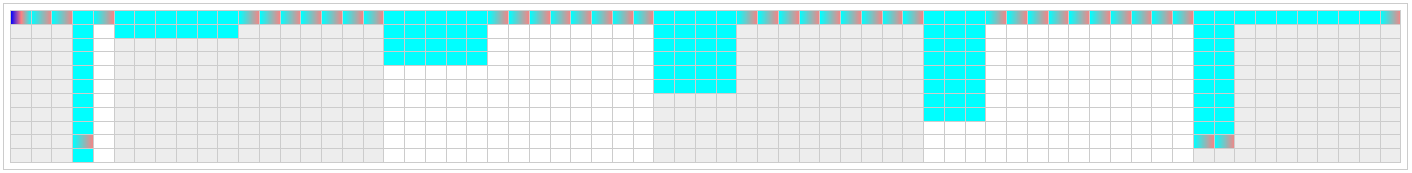}
\caption{Adding one column per column-block. \label{fig:column-shades}}
\end{figure} 
Figure~\ref{fig:row-shades} illustrates blocks that we will call row-blocks. The second image is obtained from the first by adding one row to the top of each row-block. The first image represents $S(10,3)$, the second represents $S^{(1,0)}(10,3)$.

\begin{figure}
\includegraphics[width=\textwidth]{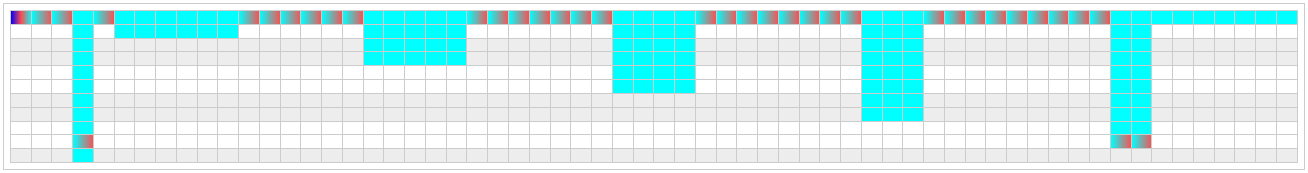}
\includegraphics[width=\textwidth]{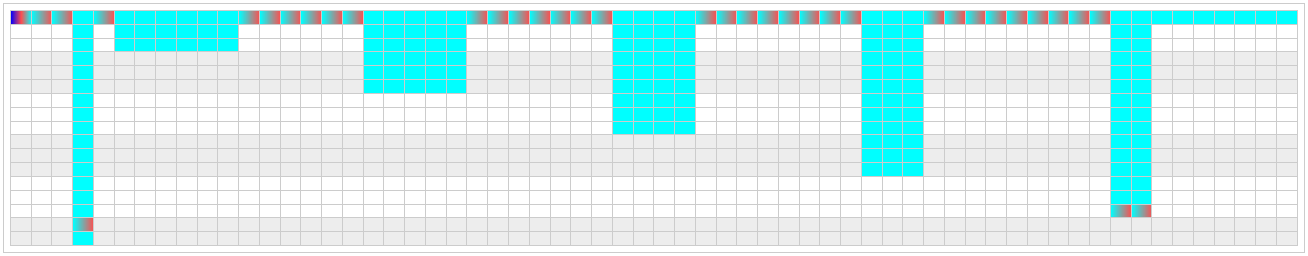}
\caption{Adding one row per row-block. \label{fig:row-shades}}
\end{figure} 

One can combine the addition of rows and of columns, as illustrated by Figure~\ref{fig:row-column-shades}. It represents $S^{(1,1)}(10,3)$.
\begin{figure}
\includegraphics[width=\textwidth]{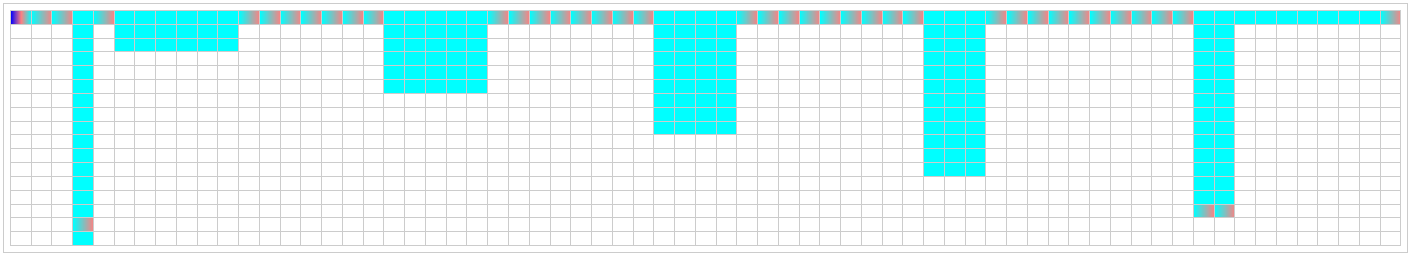}
\caption{Adding one row per row-block and one column per column-block. \label{fig:row-column-shades}}
\end{figure} 

The following remarks can be proven in a straightforward way. Having in mind the idea behind the construction may be of help.
\begin{remark}\label{rem:conductor-Sij+pr}
 $c^{(i,j+1)} = c^{(i,j)} + \frac{p^2}{2}$.
\end{remark}
\begin{remark}\label{rem:conductor-Si+jpr}
 $c^{(i+1,j)} = c^{(i,j)} + \left(\frac{p}{2}+1\right)m^{(i,j)}=c^{(i,j)} + \left(\frac{p}{2}+1\right)\left(m+ j\frac{p}{2}\right)$.
\end{remark}

Combining appropriately the previous remarks we get the following proposition.
\begin{proposition}\label{prop:conductor-Sijpr}
  The conductor of $S^{(i,j)}(p,\tau)$ is:
  \begin{align*}
  \conductoroper(S^{(i,j)}(p,\tau))
  &=\conductoroper(S(p,\tau))+j\frac{p^2}{2} + i\left(\frac{p}{2}+1\right)\left(m+ j\frac{p}{2}\right).
\end{align*}
\end{proposition}

\subsection{A visual way to look at $S^{(i,j)}(p,\tau)$}\label{subsec:visual-Spirtau}

The shape of $S^{(i,j)}(p,\tau)$ behaves as follows: when $i$ increases (the number of rows increases), the figure gets higher;  when $j$ increases (the number of columns increases), the figure gets wider.
This fact is illustrated by Figure~\ref{fig:S50100ANDS05100}, which represents the semigroups obtained from the same semigroup after adding rows and columns and drawn at the same scale.

\begin{figure}
\includegraphics[scale=1]{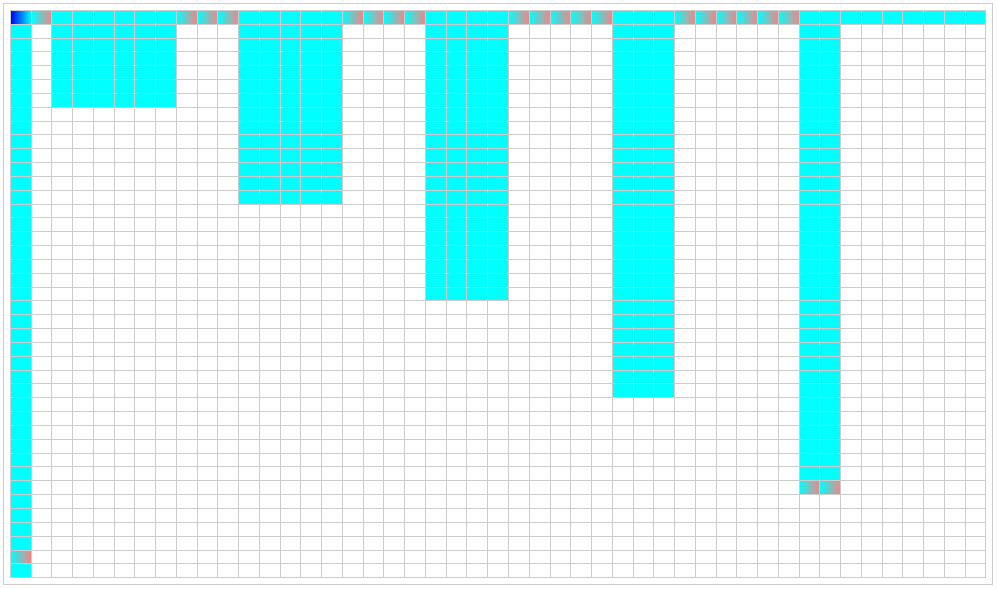}
\includegraphics[scale=1]{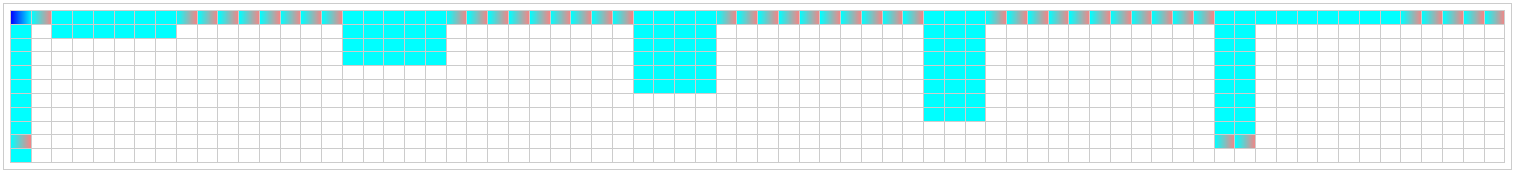}
\caption{Shape of the semigroups $S^{(5,0)}(10,0)$ and $S^{(0,5)}(10,0)$\label{fig:S50100ANDS05100}}
\end{figure} 
In a different scale, Figure~\ref{fig:S55100} illustrates the shape of the semigroup $S^{(5,5)}(10,0)$ obtained by adding the same number rows and columns to $S(10,0)$.
\begin{figure}
\includegraphics[width=\textwidth]{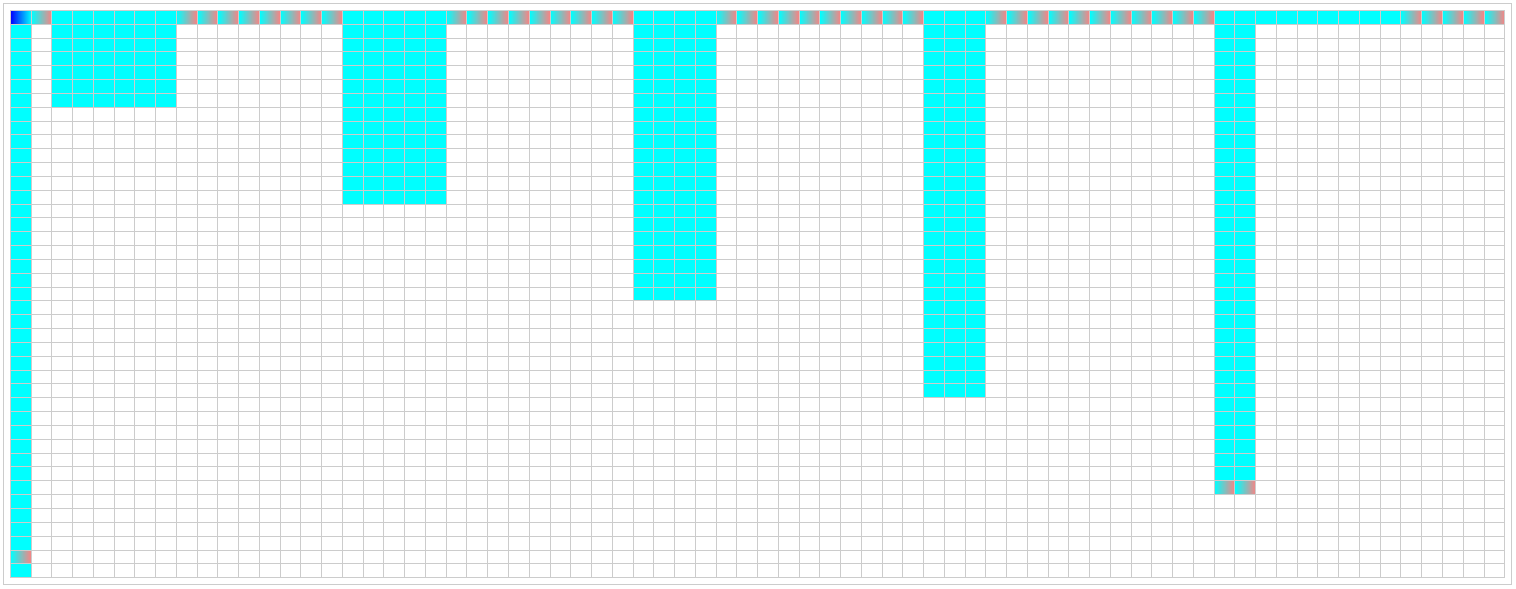}
\caption{Shape of the semigroup $S^{(5,5)}(10,0)$\label{fig:S55100}}
\end{figure} 

\subsection{Tables with numbers related to $S^{(i,j)}(p,\tau)$}\label{subsec:table-Sijptau}
Tables~\ref{table:examples-S05ptau},~\ref{table:examples-S50ptau} and~\ref{table:examples-S55ptau}, collect information related to $S^{(i,j)}(p,\tau)$ and can be compared to those tables of Sections~\ref{subsec:table-Sp} and~\ref{subsec:table-Sptau}.

\begin{table}
$\begin{array}{|l|c|c|c|c|}
\hline
\text{Semigroup}&\eliahouoper&\wilfoper&\langle \text{ left generators }\rangle_c&\text{genus}\\ \hline\hline
S^{(0,5)}(2,3)&3&21&\langle 15,17,18\rangle_{27}&23
\\\hline
S^{(0,5)}(4,3)&2&182&\langle 30,46,47\rangle_{117}&104
\\\hline
S^{(0,5)}(6,3)&0&736&\langle 47,79,80\rangle_{279}&250
\\\hline
S^{(0,5)}(8,3)&-3&2067&\langle 66,116,117\rangle_{525}&471
\\\hline
S^{(0,5)}(10,3)&-7&4713&\langle 87,157,158\rangle_{867}&777
\\\hline
S^{(0,5)}(12,3)&-12&9386&\langle 110,202,203\rangle_{1317}&1178
\\\hline
S^{(0,5)}(14,3)&-18&16992&\langle 135,251,252\rangle_{1887}&1684
\\\hline
S^{(0,5)}(16,3)&-25&28651&\langle 162,304,305\rangle_{2589}&2305
\\\hline
S^{(0,5)}(18,3)&-33&45717&\langle 191,361,362\rangle_{3435}&3051
\\\hline
S^{(0,5)}(20,3)&-42&69798&\langle 222,422,423\rangle_{4437}&3932
\\\hline
\end{array}$
\caption{Numbers related to $S^{(0,5)}(p,3)$}
\label{table:examples-S05ptau}
\end{table}

\begin{table}
$\begin{array}{|l|c|c|c|c|}
\hline
\text{Semigroup}&\eliahouoper&\wilfoper&\langle \text{ left generators }\rangle_c&\text{genus}\\ \hline\hline
S^{(5,0)}(2,3)&3&51&\langle 10,62,63\rangle_{117}&93
\\\hline
S^{(5,0)}(4,3)&2&442&\langle 20,131,132\rangle_{377}&314
\\\hline
S^{(5,0)}(6,3)&0&1751&\langle 32,214,215\rangle_{829}&700
\\\hline
S^{(5,0)}(8,3)&-3&4897&\langle 46,311,312\rangle_{1515}&1286
\\\hline
S^{(5,0)}(10,3)&-7&11213&\langle 62,422,423\rangle_{2477}&2107
\\\hline
S^{(5,0)}(12,3)&-12&22516&\langle 80,547,548\rangle_{3757}&3198
\\\hline
S^{(5,0)}(14,3)&-18&41177&\langle 100,686,687\rangle_{5397}&4594
\\\hline
S^{(5,0)}(16,3)&-25&70191&\langle 122,839,840\rangle_{7439}&6330
\\\hline
S^{(5,0)}(18,3)&-33&113247&\langle 146,1006,1007\rangle_{9925}&8441
\\\hline
S^{(5,0)}(20,3)&-42&174798&\langle 172,1187,1188\rangle_{12897}&10962
\\\hline
\end{array}$
\caption{Numbers related to $S^{(0,5)}(p,3)$}
\label{table:examples-S50ptau}
\end{table}

\begin{table}
$\begin{array}{|l|c|c|c|c|}
\hline
\text{Semigroup}&\eliahouoper&\wilfoper&\langle \text{ left generators }\rangle_c&\text{genus}\\ \hline\hline
S^{(5,5)}(2,3)&3&111&\langle 15,92,93\rangle_{177}&153
\\\hline
S^{(5,5)}(4,3)&2&882&\langle 30,196,197\rangle_{567}&504
\\\hline
S^{(5,5)}(6,3)&0&3296&\langle 47,314,315\rangle_{1219}&1090
\\\hline
S^{(5,5)}(8,3)&-3&8817&\langle 66,446,447\rangle_{2175}&1946
\\\hline
S^{(5,5)}(10,3)&-7&19463&\langle 87,592,593\rangle_{3477}&3107
\\\hline
S^{(5,5)}(12,3)&-12&37876&\langle 110,752,753\rangle_{5167}&4608
\\\hline
S^{(5,5)}(14,3)&-18&67392&\langle 135,926,927\rangle_{7287}&6484
\\\hline
S^{(5,5)}(16,3)&-25&112111&\langle 162,1114,1115\rangle_{9879}&8770
\\\hline
S^{(5,5)}(18,3)&-33&176967&\langle 191,1316,1317\rangle_{12985}&11501
\\\hline
S^{(5,5)}(20,3)&-42&267798&\langle 222,1532,1533\rangle_{16647}&14712
\\\hline
\end{array}$
\caption{Numbers related to $S^{(5,5)}(p,3)$}
\label{table:examples-S55ptau}
\end{table}

\subsection{Eliahou and Wilf numbers for $S^{(i,j)}(p,\tau)$}\label{subsec:Eliahou-Wilf-Sijptau}

Let us collect some facts that will be used to compute Eliahou and Wilf numbers.

Some of the numbers involved are invariant under row changes (changing $i$) and some others are invariant under column changes (changing $j$). We register it in the following easy or straightforward remarks.

The number of primitives is invariant under column changes:
\begin{remark}\label{rem:invariants-row-changes-S0jptau}
 Let $i$ and $j$ be non-negative integers. Then
\begin{itemize} 
\item $\lvert \primitivesoper(S^{(i,j)}(p,\tau))\rvert= \lvert \primitivesoper(S^{(0,j)}(p,\tau))\rvert$
\end{itemize}
\end{remark}
The $\qoper$-number, the number of left elements and the number of decomposables at level $\qoper$,
are invariant under column changes:
\begin{remark}\label{rem:invariants-column-changes-Si0ptau}
 Let $i$ and $j$ be non-negative integers. Then
\begin{itemize} 
\item $\qoper(S^{(i,j)}(p,\tau)) =\qoper(S^{(i,0)}(p,\tau))$;
\item $\lvert\leftsoper(S^{(i,j)}(p,\tau))\rvert= \lvert\leftsoper(S^{(i,0)}(p,\tau))\rvert$;
\item $\lvert \Dqoper(S^{(i,j)}(p,\tau))\rvert= \lvert \Dqoper(S^{(i,0)}(p,\tau))\rvert$.
\end{itemize}
\end{remark}

\begin{lemma}\label{lemma:q-number-Sijpr}
  Let $i$ and $j$ be non-negative integers. Then
$$\qoper(S^{(i,j)}(p,\tau)) =\qoper(S(p,\tau))+i\left(\frac{p}{2}+1\right).$$
\end{lemma}
\begin{proof}
By Remark~\ref{rem:invariants-column-changes-Si0ptau}, $\qoper(S^{(i,j)}(p,\tau)) =\qoper(S^{(i,0)}(p,\tau))$. Furthermore, $S^{(k+1,0)}(p,\tau)$ is the numerical semigroup whose shape is the one obtained from that of $S^{(k,0)}(p,\tau)$ by adding one row to each row-block. The number of rows added is $\frac{p}{2}+1$. The result follows by induction.
\end{proof}

\begin{lemma}\label{lemma:rho-Sijpr}
Let $i$ and $j$ be non-negative integers. Then
$$\rho(S^{(i,j)}(p,\tau)) =\tau.$$
\end{lemma}
\begin{proof}We use the fact $\rhooper(S(p)) =\tau$ (Corollary~\ref{cor:q-number-Sp}).
\begin{align}
\rho(S^{(i,j)}(p,\tau)) &= \qoper(S^{(i,j)}(p,\tau))m^{(i,j)}-c^{(i,j)}\nonumber\\
&=\left(\qoper(S(p))+i\left(\frac{p}{2}+1\right)\right)m^{(i,j)} - \left( c + j\frac{p^2}{2} + i\left(\frac{p}{2}+1\right)m^{(i,j)} \right)\nonumber\\
&=pm^{(i,j)} - \left( c + j\frac{p^2}{2} \right) = p\left( m+ j\frac{p}{2} \right) - \left( c + j\frac{p^2}{2} \right)\nonumber\\
&=pm-c=\tau \nonumber\qedhere
\end{align}
\end{proof}

\begin{lemma}\label{lemma:LSijpr}
Let $i$ and $j$ be non-negative integers. Then
$$\lvert\leftsoper(S^{(i,j)}(p,\tau))\rvert=\lvert\leftsoper(S(p,\tau))\rvert+\frac{i}{48}(p^3+12p^2+44p+48).$$
\end{lemma}
\begin{proof}
By Remark~\ref{rem:invariants-column-changes-Si0ptau}, we have $\lvert\leftsoper(S^{(i,j)}(p,\tau))\rvert =\lvert\leftsoper(S^{(i,0)}(p,\tau))\rvert$. Let us now find a formula for $\lvert \leftsoper(S^{(i,0)}(p,\tau))\rvert$.

 Let $1\le k\le \frac{p}{2}+1$. The $k$-th row added, counting upwards, contains $1+\cdots+k=\sum_{k=1}^{\frac{p}{2}+1}\frac{k(k+1)}{2}$ left elements. Summing up we get:
  \begin{align}
   \lvert \leftsoper(S^{(i+1,0)}_r(p,\tau))\rvert -\lvert\leftsoper(S^{(i,0)}_r(p,\tau))\rvert&=\sum_{k=1}^{\frac{p}{2}+1}\frac{k(k+1)}{2}\nonumber \\
    &=\frac{1}{2}\left(\frac{(p/2+1)(p/2+2)}{2}+\frac{(p/2+1)(p/2+2)(p+3)}{6}\right)\nonumber \\
    &=\frac{1}{16}(p+2)(p+4)(p/3+2)\nonumber \\
    &=\frac{1}{48}(p+2)(p+4)(p+6)\nonumber \\
   &=\frac{1}{48}(p^3+12p^2+44p+48)\nonumber
 \end{align}
By induction, we get $\lvert\leftsoper(S^{(i,0)}(p,\tau))\rvert=\lvert\leftsoper(S(p,\tau))\rvert+\frac{i}{48}(p^3+12p^2+44p+48).$
\end{proof}
\begin{corollary}\label{cor:LSijpr}
Let $i$ and $j$ be non-negative integers. Then
$$\lvert\leftsoper(S^{(i,j)}(p,\tau))\rvert=\frac{i+2}{48}p^3+\frac{i+ 18}{48}p^2 + \frac{44i+52}{48}p+i.$$
\end{corollary}
\begin{proof}
By Remark~\ref{rem:number_lefts-Sp-Sptau} and Corollary~\ref{cor:number_left_elts}, $\lvert\leftsoper(S(p,\tau))\rvert= \frac{p^3}{24}+\frac{3}{8}p^2 + \frac{13}{12}p$. It follows that 
\begin{align}
\lvert\leftsoper(S^{(i,j)}(p,\tau))\rvert&=\frac{p^3}{24}+\frac{3}{8}p^2 + \frac{13}{12}p+\frac{i}{48}(p^3+12p^2+44p+48)\nonumber \\
&=\frac{i+2}{48}p^3+\frac{i+ 18}{48}p^2 + \frac{44i+52}{48}p+i\nonumber\qedhere
\end{align}
\end{proof}

\begin{lemma}\label{lemma:PSijpr}
Let $i$ and $j$ be non-negative integers. Then
$$\lvert\primitivesoper(S^{(i,j)}(p,\tau))\rvert=\lvert\primitivesoper(S(p,\tau))\rvert+j\frac{p}{2}.$$
\end{lemma}
\begin{proof}
Just note that the projection of the elements of $\leftsoper(S^{(i,j)}(p,\tau))$ into $[c^{(i,j)},c^{(i,j)}+m^{(i,j)}[$ has the same number of elements of the projection of  $\leftsoper(S^{(r,s)}(p,\tau))$ into $[c^{(r,s)},c^{(r,s)}+m^{(r,s)}[$. The result follows from the fact that  $m^{(i,j)} = m+j\frac{p}{2}$.
\end{proof}

The number of decomposable elements at level $\qoper$ remains unchanged.
\begin{lemma}\label{lemma:DqSijpr}
Let $i$ and $j$ be non-negative integers. Then
$$\lvert \Dqoper (S^{(i,j)}(p,\tau))\rvert =\lvert \Dqoper (S(p,\tau))\rvert.$$
\end{lemma}
\begin{proof}
Let $i$ and $j$ be non-negative integers. Then
\begin{align*}
\lvert \Dqoper (S^{(i,j)}(p,\tau))\rvert &= m^{(i,j)}- \left(\lvert\primitivesoper(S(p,\tau))\rvert+j\frac{p}{2}\right) -3\\
&=m+j\frac{p}{2}- \left(\lvert \primitivesoper(S(p,\tau))\rvert+j\frac{p}{2}\right) -3=\lvert \Dqoper (S(p,\tau))\rvert.\qedhere
\end{align*}
\end{proof}

Observe that $S^{(i,j)}(p,\tau)=S^{(r,s)}(p,\tau)$ if and only if $i=r$ and $j=s$. In fact, if $i\ne r$ then the multiplicities of $S^{(i,j)}(p,\tau)$ and $S^{(r,s)}(p,\tau)$ are different; if $j\ne s$ then the conductors of $S^{(i,j)}(p,\tau)$ and $S^{(r,s)}(p,\tau)$ are different.
In particular, we can make the following remark which applies to some nice subfamilies of $\left\{S^{(i,j)}(p,\tau)\mid i,j\in \mathbb N\right\}$.
\begin{remark}\label{rem:infinite_sequence}
For fixed $p$ and $\tau$, there are infinitely many semigroups of the form $S^{(i,0)}(p,\tau)$. The same happens for the semigroups of the form $S^{(0,j)}(p,\tau)$ and $S^{(j,j)}(p,\tau)$. 
\end{remark}

\begin{theorem}\label{th:same_w0-Sijptau}
Let $p$ be an even positive integer and let $\tau$ be a non-negative integer. For non-negative integers $i,j$, the following holds: 
$$\eliahouoper(S^{(i,j)}(p,\tau))=\eliahouoper(S(p,\tau)).$$
\end{theorem}
\begin{proof}
Again, the number of left primitives is $3$, that is, $\lvert \primitivesoper(S^{(i,j)}(p,\tau))\cap \leftsoper(S^{(i,j)}(p,\tau))\rvert=3$.
Next we use the previously proved results.
\begin{align*}
&\eliahouoper(S^{(i,j)}(p,\tau))=3 \lvert\leftsoper(S^{(i,j)}(p,\tau)) \rvert - \qoper(S^{(i,j)}(p,\tau)) \lvert \Dqoper(S^{(i,j)}(p,\tau))\rvert +\rho(S^{(i,j)}(p,\tau))\\
&\quad\quad =3 \left(\lvert\leftsoper(S(p,\tau))\rvert+\frac{i}{48}(p^3+12p^2+44p+48)\right) - \left(\qoper(S(p,\tau))+i\left(\frac{p}{2}+1\right)\right) \lvert \Dqoper(p,\tau))\rvert +\tau\\
&\quad\quad =\left(3 \lvert\leftsoper(S(p,\tau))\rvert - \qoper(S(p,\tau))\lvert \Dqoper(p,\tau))\rvert +\tau\right) + \frac{i}{16}(p^3+12p^2+44p+48)\\
&\quad\quad\hspace{10cm} - i\left(\frac{p}{2}+1\right) \lvert \Dqoper(p,\tau))\rvert.
\end{align*}
It remains to observe that $\left(\frac{p}{2}+1\right) \lvert \Dqoper(p,\tau))\rvert=\frac{1}{16}(p^3+12p^2+44p+48)$, which is straightforward:\\
$\left(\frac{p}{2}+1\right)\left(\frac{p^2}{4}+\frac{5}{2}p+6\right)/2=
\left(\frac{p^3}{8} + \frac{6p^2}{4} +\frac{11}{2}p  +6  \right)/2=
\frac{p^3}{16} + \frac{3p^2}{4} +\frac{11}{4}p  +3$
\end{proof}
The proof of the following theorem is lengthy, but straightforward. We leave it to Appendix~\ref{appendix:proof-wilf-Sijptau}.
\begin{theorem}\label{th:wilf-Sijptau}
Let $p$ be an even positive integer and let $\tau$, $i$ and $j$ be non-negative integers. Then 
\begin{align*}
\wilfoper(S^{(i,j)}(p,\tau))&=\frac{p^5}{192}+\frac{5p^4}{64}+\frac{p^3}{4}-\frac{7p^2}{16}+\frac{p}{6}+ \tau\cdot\left(\frac{p^4}{48}+\frac{3}{16}p^3 + \frac{1}{24}p^2 + 1\right)\\
&+i\cdot\left(\frac{p^5}{384}+\frac{3p^4}{64}+\frac{7p^3}{32}+\frac{p^2}{16}-\frac{5p}{12}\right) +  j\cdot\left(\frac{p^4}{48}+\frac{3p^3}{16}+\frac{p^2}{24}\right)\\
&+ij\cdot\left(\frac{p^4}{96}+\frac{p^3}{8}+\frac{5p^2}{24}\right) +  i\tau\cdot\left(\frac{p^4}{96}+\frac{p^3}{8}+\frac{5p^2}{24}\right).
\end{align*}
\end{theorem}

There are various consequences that we can now state. It is immediate from Theorem~\ref{th:wilf-Sijptau} that Wilf's conjecture holds for every numerical semigroup of the form $S^{(i,j)}(p,\tau)$, as the following corollary says.
\begin{corollary}\label{cor:wilf_holds_for_the_family}
Let $p$ be an even positive integer and let $\tau$ be a non negative integer. Then $\wilfoper (S^{(i,j)}(p,\tau)) > 0$, for every non-negative integers $i,j$.
\end{corollary}

Another interesting consequence is given by the following corollary.

\begin{corollary}\label{cor:large-Wilf}
 Given integers $n$ and $N$, there are infinitely many numerical semigroups~$S$ such that $\eliahouoper(S)=n$ and $\wilfoper(S)>N$. 
\end{corollary}
\begin{proof}
  By Corollary~\ref{cor:every_integer_is_Eliahou_number}, there exists a numerical semigroup $S(p,\tau)$ such that $\eliahouoper(S(p,\tau))=n$. Let us now consider the family $(S^{(0,j)}(p,\tau))_{j\in\mathbb N}$ of numerical semigroups.
  Let $k_0$ be an integer such that $n+k_0\frac{p^2(p^2+9p+2)}{48} >N$ (which clearly exists). Then, using Theorem~\ref{th:wilf-Sijptau}, $\eliahouoper(S^{(0,k)}(p,\tau))=n+k_0\frac{p^2(p^2+9p+2)}{48}>N$, whenever $k>k_0$.
\end{proof}

\appendix

\section{A remark concerning the genus}\label{sec:remark-genus}
Let $n$ be an integer. 
Consider the set $\mathcal S_n$ of numerical semigroups with Eliahou number $n$:
$$\mathcal S_n = \left\{S\mid S \text{ is a numerical semigroup such that }\eliahouoper(S)=n\right\}.$$

In Section~\ref{sec:Sijptau} we proved that all the $\mathcal S_n$'s consist of infinitely many numerical semigroups.
As a consequence, we have that, for all $n$, the set $$GE_n = \left\{\genusoper(S)\mid S \in \mathcal S_n\right\}$$ of positive integers is, in particular, non empty and therefore has a minimum. Furthermore, since there are only finitely many numerical semigroups with a given genus, none of the $GE_n$ is bounded (see below for an alternative proof).

We remark that Fromentin's computations (see Section~\ref{subsec:fromentin}) show that the minimum of $GE_{-1}$ is $43$ and that if $n<-1$, then $\min GE_{n} > 60$.

A related and general question is:
\begin{question}\label{quest:smallest-genus-eliahou-negative}
  Le $n$ be an integer. Find the minimum of $GE_n$.
\end{question}
It seems to be difficult to get some hint on how to attack this question through (presently available) computational means. Finding some good bounds is probably challenging enough. To this end, we prove the following proposition:

\begin{proposition}\label{prop:genus-Sp}
  Let $p$ be an even positive integer. Then
 \begin{align*}
  &\genusoper(S(p))= p\left(\frac{5}{24}p^2+\frac{13}{8}p+\frac{11}{12}\right)
 \end{align*}
\end{proposition}
\begin{proof}
  Using the fact that $\conductoroper(S(p))=p\mu$ (Corollary~\ref{cor:conductor-Sp}) and the formula for $\lvert \leftsoper(S(p))\rvert$ given by Corollary~\ref{cor:number_left_elts}, we get:
 \begin{align*}
 \genusoper(S(p)) & =p\mu-\lvert L\rvert\\
                  & =p\left(\frac{p^2}{4}+2p+2\right)-\left(\frac{p^3}{24}+\frac{3}{8}p^2 + \frac{13}{12}p\right)\\
		  & =p\left(\frac{5}{24}p^2+\frac{13}{8}p+\frac{11}{12}\right).\qedhere
\end{align*}
\end{proof}

One can state similar results for the others families considered, namely for $S^{(0,j)}(p,\tau)$. 
\commentgray{Since we are interested in semigroups having as small genus as possible, there is no interest in stating the general result for $S^{(i,j)}(p,\tau)$, with $i$ or $j$ possibly non-zero.}

\begin{proposition}\label{prop:genus-S0jptau}
 Let $p$ be an even positive integer and let $\tau$ and $j$ be non-negative integers. Then
\begin{align*}
  \genusoper(S^{(0,j)}(p,\tau))=\genusoper(S(p))+(\tau+j)\frac{p^2}{2}-\tau.
\end{align*}
\end{proposition}
\begin{proof}
 By Proposition~\ref{prop:conductor_Sptau} and Proposition~\ref{prop:conductor-Sijpr} (with $i=0$) we conclude that $\conductoroper(S^{(0,j)}(p,\tau))= \conductoroper(S(p))+(\tau+j)\frac{p^2}{2}-\tau$. By Remark~\ref{rem:invariants-column-changes-Si0ptau}, all the semigroups under consideration have the same number of elements: $\lvert \leftsoper(S(p)\rvert$.
\end{proof}

Since the Eliahou number of $S^{(0,j)}(p,\tau)$ does not vary with $j$ and, for any given $n$, there exist $p$ and $\tau$ such that $\eliahouoper(S(p,\tau))=n$, we have that $\left\{\genusoper(S^{(0,j)}(p,\tau))\mid j\in \mathbb N\right\}\subseteq GE_n$. We can therefore conclude, as already observed, that the set $GE_n$ is unbounded for any $n$.

\begin{corollary}\label{cor:upper-bounds-for-minimum-negative}
  Let $n$ be a negative integer and let $p=\left\lceil 1 + \sqrt{1-8n}\right\rceil$. Then the numerical semigroup $S=S\left(p,n-\frac{p}{4}(1-p/2)\right)$ is such that 
  $$\genusoper(S)=\frac{p}{4}\left(\frac{p^3}{4}+\frac{p^2}{3} +6p + \frac{14}{3}\right) + n\left(\frac{p^2}{2}-1\right).$$
  In particular, we get the following bound:
\begin{align}  
 \min GE_n \le \frac{p^4}{16}+\frac{p^3}{12} +\frac{3p^2}{2} + \frac{14p}{12} + n\left(\frac{p^2}{2}-1\right).\label{eq:upper-bounds-for-minimum-negative}
\end{align}
\end{corollary}
\begin{proof}
  Recall that, by Corollary~\ref{cor:every_negative_integer_is_Eliahou_number}, $\eliahouoper(S)=n$.
  By Proposition~\ref{prop:genus-S0jptau}, we have:
  \begin{align*}
   \genusoper(S)&=\genusoper(S(p))+\tau\left(\frac{p^2}{2}-1\right)\\
   &= p\left(\frac{5}{24}p^2+\frac{13}{8}p+\frac{11}{12}\right) + \left(n-\frac{p}{4}(1-p/2)\right)\left(\frac{p^2}{2}-1\right)\\
   &= \frac{p}{4}\left(\frac{5}{6}p^2+\frac{13}{2}p+\frac{11}{3}+\left(\frac{p}{2}-1\right)\left(\frac{p^2}{2}-1\right)\right)         + n\left(\frac{p^2}{2}-1\right)\\
   &= \frac{p}{4}\left(\frac{5}{6}p^2+\frac{13}{2}p+\frac{11}{3}+ \frac{p^3}{4} -\frac{p^2}{2} -\frac{p}{2} +1\right)         + n\left(\frac{p^2}{2}-1\right)\\
   &= \frac{p}{4}\left(\frac{p^3}{4}+\frac{p^2}{3} +6p + \frac{14}{3}\right) + n\left(\frac{p^2}{2}-1\right)\qedhere
  \end{align*}
\end{proof}
The bound given by Corollary~\ref{cor:upper-bounds-for-minimum-negative} is not tight in general, as the below discussion shows. 
\smallskip

For $n=-1$ we have $p=4$ and $n-\frac{p}{4}(1-p/2)=0$. As we know, $S(4)=\langle 14, 22, 23\rangle_{56}$ is such that $\eliahouoper(S(4))=-1$ and $\genusoper(S(4))=43$, which is a minimum, by Fromentin's computations. Evaluating $\frac{p^4}{16}+\frac{p^3}{12} +\frac{3p^2}{2} + \frac{14p}{12} + n\left(\frac{p^2}{2}-1\right)$ with $p=4$ and $n=-1$ one gets $43$. Therefore, we can state the following remark.

\begin{remark}\label{rem:tight-bound--1}
 The bound given by Corollary~\ref{cor:upper-bounds-for-minimum-negative} is tight for $n=-1$.
\end{remark}

The tightness is no longer true for $n=-2$. In this case $p=6$ and ${n-\frac{p}{4}(1-p/2)=1}$. We have $S(6,1)=\langle 26, 44, 45\rangle_{155}$, which is a semigroup of genus $126$ (one can compute it directly or evaluate the right side of (\ref{eq:upper-bounds-for-minimum-negative}) at $p=6$ and $n=-2$). But one can verify that $\langle 30, 44, 48, 49\rangle_{118}$ has genus $99$ and Eliahou number $-2$, thus $126$ is not tight.
We do not know whether $99$ is tight.
\begin{question}\label{quest:smallest-genus-eliahou--2-99}
  Is $99$ the minimum possible genus of all the numerical semigroups satisfying $\eliahouoper(S)=-2$?
\end{question}

The bound given by Corollary~\ref{cor:upper-bounds-for-minimum-negative} for $n=-3$ is smaller than the one given for $n=-2$ and we do not know whether it is tight.
For $n=-3$, we get $p=6$ and $n-\frac{p}{4}(1-p/2)=0$. The semigroup $S(6,0)=\langle 23, 39, 40\rangle_{138}$ has genus $109$ and Eliahou number $-3$.
\begin{question}\label{quest:smallest-genus-eliahou--3-109}
  Is $109$ a lower bound for the set of genus of all numerical semigroups satisfying $\eliahouoper(S)=-3$?
\end{question}
\medskip

Let us now make some comments on a possible attempt to solve Questions~\ref{quest:smallest-genus-eliahou--2-99} and~\ref{quest:smallest-genus-eliahou--3-109} by exhaustive search.

Denote by $n_g$ the number of numerical semigroups of genus $g$ and by $N_g$ the number of numerical semigroups of genus not greater than $g$. The following was conjectured by Bras-Amoros~\cite{Bras-Amoros2008SF-Fibonacci-like} and proved by Zhai~\cite{Zhai2012SF-Fibonacci}.
\begin{proposition}\label{prop:fibonacci-like}
The sequence of the number of numerical semigroups by genus behaves like the Fibonacci sequence, that is, 
$$\lim_{g\to +\infty}\frac{n_{g-2}+n_{g-1}}{n_g}=1.$$ 
\end{proposition}

This implies that $\lim_{g\to +\infty}\frac{n_{g}}{n_{g-1}}=\varphi$, where $\varphi$ is the golden ratio.
Since the golden ratio is greater than $1,61$, it is not surprising that one can experimentally observe $n_g \ge n_{g-1}\cdot(1.6)$ even for small $g$. This can be observed in the various tables available in the literature containing the numbers $n_g$. The one that, to the best of our knowledge, goes the farthest is the one given by Fromentin and Hivert~\cite{FromentinHivert} where, in particular, the number of numerical semigroups of genus $67$ is given: $n_{67}=377\, 866\, 907\, 506\, 273$. 
Summing up the number of numerical semigroups of genus not greater $67$, one gets $$N_{67}=\sum_{g\in\{1,\ldots,67\}}n_g\simeq 9.86\cdot 10^{14}.$$
One can guess that the number of numerical semigroups of genus $77+i$ is approximately $n_{109} \ge n_{67}\cdot(1.6)^{i}$ and consequently $$N_{109}=\sum_{g\in\{1,\ldots,109\}}n_g\simeq 9.86\cdot 10^{14}+n_{67}\frac{(1.6)^{43}-1.6}{1.6-1}.$$

By doing some computations (with several unfavourable rounding) one gets that the number of numerical semigroups up to genus $109$ is well over $(3.2)\cdot 10^{23}$. 
It is out of reach using today’s computational means to do even the most trivial computation for such a number of semigroups. 

Suppose that one could do some computation with $10^6$ semigroups per second. The overall computation with the semigroups of genus up to $109$ would require $(3.2)\cdot 10^{17}$ seconds. Taking into account that one year has less that $(1.6)\cdot 10^7$ seconds, the computation would take about $2\cdot 10^{10}$ years, which is a little bit too much to wait for.

Although computing the Eliahou number of one numerical semigroup of small genus, up to $109$, say, takes almost no time, the above discussion should make clear that exhaustive search is not satisfactory when one has to deal with huge numbers of numerical semigroups. Further theoretical results are needed.

To the best of our knowledge, exhaustive search has only been performed for semigroups up to genus $60$. In particular, we are not able to guarantee that $61$ is not the answer for one or both questions above.


\section{A variety of examples}\label{sec:examples}

\begin{table}
$\begin{array}{|l|c|c|c|c|}
\hline
&\eliahouoper&\wilfoper&\langle \text{ left generators }\rangle_c&\text{genus}\\ \hline\hline
1&-16&6512&\langle 122,214,217,219\rangle_{976}&872\\ \hline 
2&-16&15795&\langle 219,379,389,390\rangle_{1740}&1635\\ \hline 
3&-15&16602&\langle 247,373,375,390\rangle_{1728}&1634\\ \hline 
4&-14&5791&\langle 170,290,293,297\rangle_{1020}&971\\ \hline 
5&-14&5920&\langle 173,292,300,301\rangle_{1038}&989\\ \hline 
6&-12&18513&\langle 303,400,422,423\rangle_{1818}&1737\\ \hline 
7&-11&18868&\langle 273,413,424,432\rangle_{1906}&1812\\ \hline 
8&-10&9106&\langle 247,415,420,424\rangle_{1478}&1429\\ \hline 
9&-10&13190&\langle 206,308,311,327\rangle_{1440}&1345\\ \hline 
10&-10&13865&\langle 241,319,327,332\rangle_{1444}&1363\\ \hline 
11&-8&3690&\langle 121,200,203,210\rangle_{720}&671\\ \hline 
12&-8&7087&\langle 200,336,343,347\rangle_{1194}&1145\\ \hline 
13&-8&8807&\langle 239,411,412,420\rangle_{1434}&1385\\ \hline 
14&-8&12520&\langle 199,306,312,316\rangle_{1392}&1298\\ \hline 
15&-7&9152&\langle 247,423,426,427\rangle_{1481}&1432\\ \hline 
16&-7&18698&\langle 271,412,419,424\rangle_{1888}&1794\\ \hline 
17&-5&2704&\langle 97,162,166,167\rangle_{579}&530\\ \hline 
18&-5&4405&\langle 133,218,223,232\rangle_{797}&746\\ \hline 
19&-5&7651&\langle 143,215,218,226\rangle_{997}&903\\ \hline 
20&-5&17743&\langle 259,388,401,408\rangle_{1809}&1715\\ \hline 
21&-5&20440&\langle 291,436,455,459\rangle_{2026}&1932\\ \hline 
22&-4&206&\langle 34,52,55,56\rangle_{136}&117\\ \hline 
23&-4&221&\langle 35,53,56,58\rangle_{140}&121\\ \hline 
24&-3&207&\langle 34,52,55,56\rangle_{135}&116\\ \hline 
25&-3&297&\langle 40,61,64,65\rangle_{159}&140\\ \hline 
26&-3&417&\langle 48,73,78,79\rangle_{191}&172\\ \hline 
27&-2&148&\langle 30,44,48,49\rangle_{118}&99\\ \hline 
28&-2&298&\langle 40,61,64,65\rangle_{158}&139\\ \hline 
29&-1&359&\langle 44,67,71,72\rangle_{173}&154\\ \hline 
30&-1&404&\langle 47,73,76,77\rangle_{185}&166\\ \hline 
\end{array}$
\caption{Sporadic semigroups with  4 left primitives}
\label{table:sporadic-4-primitives}
\end{table}
\begin{table}
$\begin{array}{|l|c|c|c|c|}
\hline
&\eliahouoper&\wilfoper&\langle \text{ left generators }\rangle_c&\text{genus}\\ \hline\hline
1&-9&16931&\langle 310,517,518,526,532\rangle_{1841}&1765\\ \hline 
2&-6&1798&\langle 116,178,182,188,193\rangle_{464}&438\\ \hline 
3&-5&4703&\langle 248,376,381,399,402\rangle_{991}&965\\ \hline 
4&-2&1274&\langle 91,137,141,148,154\rangle_{364}&338\\ \hline 
5&-1&2023&\langle 125,195,201,202,204\rangle_{499}&473\\ \hline 
6&-1&2309&\langle 139,218,225,228,229\rangle_{551}&525\\ \hline 
\end{array}$
\caption{Sporadic semigroups with  5 left primitives}
\label{table:sporadic-5-primitives}
\end{table}
\begin{table}
$\begin{array}{|l|c|c|c|c|}
\hline
&\eliahouoper&\wilfoper&\langle \text{ left generators }\rangle_c&\text{genus}\\ \hline\hline
1&-383&117926113&\langle 4505,8798,8816\rangle_{409867}&374352\\ \hline 
2&-381&118025985&\langle 4505,8798,8816\rangle_{409955}&374410\\ \hline 
3&-220&117772232&\langle 4078,6349,6351\rangle_{358048}&308204\\ \hline
4&-167&23242633&\langle 2103,3373,3377\rangle_{136469}&117035\\ \hline
5&-132&16188078&\langle 1993,3029,3044\rangle_{107601}&94998\\ \hline 
6&-124&1727371&\langle 792,1507,1508\rangle_{30877}&27633\\ \hline 
7&-18&1936766&\langle 830,1499,1502\rangle_{32229}&28694\\ \hline 
8&-17&1881412&\langle 1034,1644,1651\rangle_{32032}&29658\\ \hline
9&-16&1916340&\langle 824,1487,1489\rangle_{31996}&28460\\ \hline
 10&-15&1996941&\langle 793,1258,1266\rangle_{30924}&26367\\ \hline 
11&-13&4347428&\langle 984,1765,1769\rangle_{53852}&44284\\ \hline 
12&-7&2019449&\langle 852,1537,1541\rangle_{33106}&29561\\ \hline
13&-5&2432560&\langle 880,1530,1533\rangle_{36032}&31608\\ \hline 
14&-4&2309256&\langle 865,1374,1381\rangle_{33658}&29135\\ \hline
15&-3&1597047&\langle 731,1317,1322\rangle_{28395}&24846\\ \hline
16&-2&2150683&\langle 822,1257,1259\rangle_{31973}&27101\\ \hline
17&-1&2132899&\langle 854,1475,1478\rangle_{33274}&29385\\ \hline 
\end{array}$
\caption{Sporadic semigroups with  3 left large primitives}
\label{table:sporadic-3large-primitives}
\end{table}
In this section we give a number of examples of numerical semigroups whose Eliahou number is negative. We group these examples into various tables.
The choice of the examples, most of them obtained in the process that lead to obtaining the families considered in the paper, deserves some comments. 
One of the aims is to furnish a source of (counter-) examples for some natural questions. Another one is to give some insight on possible answers to others. 
Others examples are probably mere curiosities, that, as frequently happens when doing experiments, may help to ask the appropriate questions. 

How general are the families we considered in the paper? 
We believe that not much: its contribution to the characterization of $\mathcal{E}$ is probably very small. 
We present examples having more than $3$ minimal left generators as well as examples where none of its pairs of minimal left generator consists of consecutive numbers. None of these examples is of the form $S^{(i,j)}(p,\tau)$.  

The tables presented in this appendix are similar to the ones that already appeared in the text, namely in Section~\ref{subsec:table-Sp} where an explanation for the meaning of the various columns is given. The only difference occurs in the first column: instead of the notation used for the semigroup, a number is used so that the examples in the table are numbered sequentially.

Several of our examples have more that $3$ minimal left generators (see Tables~\ref{table:sporadic-4-primitives} and~\ref{table:sporadic-5-primitives}. Thus we can state the following remark.
\begin{remark}\label{remark:left_gens}
There exist numerical semigroups with more than $3$ minimal left generators whose Eliahou number is negative.
\end{remark}
Our examples just involve numerical semigroups with no more than $5$ minimal left generators, but we would not be surprised if the following question has a negative answer. 
\begin{question}\label{quest:left_gens}
Is there an upper bound for the number of left primitives of the numerical semigroups in $\mathcal E$?
\end{question}


\section{Wilf number of $S^{(i,j)}(p\tau)$}\label{appendix:proof-wilf-Sijptau}
The aim of this section is to prove Theorem~\ref{th:wilf-Sijptau}, which gives a formula in terms of $i,j,p$ and $\tau$ for the Wilf number of $S^{(i,j)}(p,\tau)$.
\begin{proof}
By definition, 
\begin{align*}
\wilfoper(S^{(i,j)}(p,\tau))&=\lvert \primitivesoper(S^{(i,j)}(p,\tau))\rvert \lvert\leftsoper(S^{(i,j)}(p,\tau)) \rvert - \conductoroper(S^{(i,j)}(p,\tau)).
\end{align*}
Using Corollary~\ref{cor:LSijpr}, Lemma~\ref{lemma:PSijpr} and Proposition~\ref{prop:conductor-Sijpr}, we get
\begin{align*}
\wilfoper(S^{(i,j)}(p,\tau))&=\left(\lvert\primitivesoper(S(p,\tau))\rvert+j\frac{p}{2}\right) \left(\lvert\leftsoper(S(p,\tau))\rvert+\frac{i}{48}(p^3+12p^2+44p+48)\right)\\
&\quad\quad - \left(\conductoroper(S(p,\tau))+j\frac{p^2}{2} + i\left(\frac{p}{2}+1\right)m^{(i,j)}\right)
\end{align*}
Through small manipulations we get:
\begin{align}
\wilfoper(S^{(i,j)}(p,\tau))&=\wilfoper(S(p,\tau))\nonumber\\
&\quad\quad + \frac{i}{48}\lvert\primitivesoper(S(p,\tau))\rvert(p^3+12p^2+44p+48)\label{parcel1}\\
&\quad\quad + j\frac{p}{2}\left(\lvert\leftsoper(S(p,\tau))\rvert+\frac{i}{48}(p^3+12p^2+44p+48)\right)\label{parcel2}\\
&\quad\quad -\left(j\frac{p^2}{2} + i\left(\frac{p}{2}+1\right)\left(\frac{p^2}{4}+(\tau+j+4)\frac{p}{2}+2\right)\right)\label{parcel3}
\end{align}

Now let us look separately to each of the numbered parcels.
To deal with Parcel~(\ref{parcel1}) we need to use Remark~\ref{rem:number_primitives-Sptau}.
\begin{align}
 &\frac{i}{48}\lvert\primitivesoper(S(p,\tau))\rvert(p^3+12p^2+44p+48)\nonumber\\
&\quad\quad = \frac{i}{8\cdot 48}\left(p^2+(6 + 4\tau) p+16\right)\left(p^3+12p^2+44p+48\right)\nonumber\\
&\quad\quad =\frac{i}{8\cdot 48}\left( p^5+ (12+(6+4\tau))p^4+(60+12(6+4\tau))p^3\right)\nonumber\\
&\quad\quad  + \frac{i}{8\cdot 48}\left((5\cdot 48+44(6+4\tau))p^2+(44\cdot 16+48(6+4\tau))p+48\cdot 16\right)\nonumber\\
&\quad\quad =\frac{i}{8\cdot 48}\left( p^5+ (18+4\tau)p^4+(11\cdot 12+48\tau)p^3\right)\nonumber\\
&\quad\quad  + \frac{i}{8\cdot 48}\left((21\cdot 24+(11\cdot 16)\tau)p^2+(62\cdot 16 +48\cdot 4\tau))p+16\cdot 48\right)\nonumber\\
\nonumber
\end{align}
To deal with Parcel~(\ref{parcel2}) we need to use Remark~\ref{rem:number_lefts-Sp-Sptau} (see also Corollary~\ref{cor:number_left_elts}).
\begin{align}
&j\frac{p}{2}\left(\lvert\leftsoper(S(p,\tau))\rvert+\frac{i}{48}(p^3+12p^2+44p+48)\right)\nonumber\\
&\quad\quad = j\frac{p}{2}\left(\frac{p^3}{24}+\frac{3}{8}p^2 + \frac{13}{12}p+\frac{i}{48}(p^3+12p^2+44p+48)\right)\nonumber\\
&\quad\quad = \frac{j}{2\cdot 48}\left((i+2)p^4+(12i+ 18)p^3 + (44i+52)p^2+48ip\right)\nonumber\\
\nonumber
\end{align}
Next we rewrite Parcel~(\ref{parcel3}):
\begin{align}
&-\left(j\frac{p^2}{2} + i\left(\frac{p}{2}+1\right)\left(\frac{p^2}{4}+(\tau+j+4)\frac{p}{2}+2\right)\right)\nonumber\\
&\quad\quad = -j\frac{p^2}{2} - \frac{i}{8}\left(p+2\right)\left(p^2+(8+2(\tau+j))p+8\right)\nonumber\\
&\quad\quad = -j\frac{p^2}{2} -\frac{i}{8}\left(p^3+(10+2(\tau+j))p^2+(24+4(\tau+j))p+16\right)\nonumber
\end{align}

Joining up, we get:
\begin{align*}
\wilfoper(S^{(i,j)}(p,\tau))&=\wilfoper(S(p,\tau))\\
&+ \left(\frac{i}{8\cdot 48}      \right)p^5 \\
&+ \left(\frac{i}{8\cdot 48}(18+4\tau) + \frac{j}{2\cdot 48}(i+2) \right)p^4 \\
&+ \left(\frac{i}{8\cdot 48}(11\cdot 12+48\tau)   + \frac{j}{2\cdot 48}(12i+ 18)  -\frac{i}{8}\right)p^3 \\
&+ \left(\frac{i}{8\cdot 48}(21\cdot 24+(11\cdot 16)\tau)+ \frac{j}{2\cdot 48}(44i+52)-\frac{i}{8}(10+2\tau+2j) -\frac{j}{2}    \right)p^2 \\
&+ \left(\frac{i}{8\cdot 48} (62\cdot 16 +48\cdot 4\tau) + \frac{j}{2\cdot 48}48i -\frac{i}{8}(24+4(\tau+j))  \right)p \\
&+ \left(\frac{i}{8\cdot 48}16\cdot 48  -\frac{i}{8} 16  \right). 
\end{align*}
We proceed by slightly manipulating the right side of the above equality.
\begin{align*}
\wilfoper(S^{(i,j)}(p,\tau))&=\wilfoper(S(p,\tau))+ \left(\frac{i}{8\cdot 48}      \right)p^5 \\
&+ \left(3i/64+j/48 +   ij/96+ i\tau/96\right)p^4 \\
&+ \left(7i/32 + 3j/16+   ij/8 +i\tau/8\right)p^3 \\
&+ \left(i(21/16+11/24\tau)+ 11ij/24 + 13j/24 -5i/4  -i\tau/4-ij/4 -j/2  \right)p^2 \\
&+ i\left(31/12 +\tau/2 + j/2 -3-\tau/2-j/2  \right)p \\
&+ \left(2i  - 2i  \right).
\end{align*}
Simplifying a bit further:
\begin{align*}
\wilfoper(S^{(i,j)}(p,\tau))&=\wilfoper(S(p,\tau))+ \left(\frac{i}{8\cdot 48}      \right)p^5 \\
&=\wilfoper(S(p,\tau))+ \left(\frac{i}{8\cdot 48}      \right)p^5 \\
&+ \left(3i/64+j/48 +   ij/96+ i\tau/96\right)p^4 \\
&+ \left(7i/32 + 3j/16+   ij/8 +i\tau/8\right)p^3 \\
&+ \left( i/16 +j/24 +5ij/24 + 5i\tau/24  \right)p^2 \\
 & -\frac{5i}{12}p.
\end{align*}
Next we use Proposition~\ref{prop:wilf_holds_for_all_p_rho} to obtain the requested expression.
\begin{align*}
\wilfoper(S^{(i,j)}(p,\tau))&=\frac{p^5}{192}+\frac{5p^4}{64}+\frac{p^3}{4}-\frac{7p^2}{16}+\frac{p}{6}+ \tau\cdot\left(\frac{p^4}{48}+\frac{3}{16}p^3 + \frac{1}{24}p^2 + 1\right)\\
&+i\cdot\left(\frac{p^5}{384}+\frac{3p^4}{64}+\frac{7p^3}{32}+\frac{p^2}{16}-\frac{5p}{12}\right) +  j\cdot\left(\frac{p^4}{48}+\frac{3p^3}{16}+\frac{p^2}{24}\right)\\
&+ij\cdot\left(\frac{p^4}{96}+\frac{p^3}{8}+\frac{5p^2}{24}\right) +  i\tau\cdot\left(\frac{p^4}{96}+\frac{p^3}{8}+\frac{5p^2}{24}\right).\qedhere
\end{align*}
\end{proof}

\subsection*{Acknowledgments} 
Besides the formal acknowledgement to the Centre for Mathematics of the University of Porto (CMUP) made somewhere in this paper, I would like to thank CMUP for providing me the computational tools necessary to this kind of works.\\
I would like also to thank Shalom Eliahou for his suggestions and comments on preliminary versions. Special thanks to Shalom for encouraging me to publish this work.

\end{document}